\documentclass[12pt, twoside]{article}
\usepackage[final]{graphicx}
\usepackage{amsmath}
\usepackage{amsthm}
\usepackage{amssymb}
\usepackage{hyperref}
\usepackage{mathtools}
\usepackage{comment}
\usepackage{subfigure}
\usepackage{times}

\usepackage{verbatim}
\theoremstyle{definition}
\newtheorem{theorem}{Theorem}[section]
\newtheorem{corollary}[theorem]{Corollary}
\newtheorem{lemma}[theorem]{Lemma}
\newtheorem{proposition}[theorem]{Proposition}
\newtheorem{definition}[theorem]{Definition}
\newtheorem{remark}[theorem]{Remark}
\newtheorem{problem}[theorem]{Problem}
\newtheorem{assumption}[theorem]{Assumption}
\newtheorem{example}[theorem]{Example}

\numberwithin{equation}{section}

\newcommand{\subjclass}[1]{\bigskip\noindent\emph{2010 Mathematics Subject Classification:}\enspace#1}
\newcommand{\keywords}[1]{\noindent\emph{Keywords:}\enspace#1}

\newcommand{\Div}{{\rm div}}
\newcommand{\curl}{{\rm curl}}
\newcommand{\R}{\mathbb{R}}
\newcommand{\g}{{\Gamma_0}}

\DeclareMathOperator*{\argmin}{arg\,min}
\newcommand{\divg}{{\rm div}_\Gamma}
\newcommand{\id}[1]{\,{\rm id}_{#1}\,}
\newcommand{\D}[1]{\underline{D}_{#1}}

\newcommand{\cee}{\mathcal{C}}
\newcommand{\dee}{{\rm d}}
\newcommand{\Jstar}{J^*}
\newcommand{\Estar}{\mathcal{E}^*}
\newcommand{\I}{{\rm I}}

\begin{document}
\baselineskip=17pt
\title{A formula for membrane mediated point particle interactions on near spherical biomembranes}
%{The membrane mediated force on point particles attached with application to a near spherical biomembrane}
%\author{C.M. Elliott and P.J. Herbert}
%\author[C.~M.~Elliott and P.~J.~Herbert]{C\ls H\ls A\ls R\ls L\ls E\ls S\ns M.\ns E\ls L\ls L\ls I\ls O\ls T\ls T\ns
%	\and
%	P\ls H\ls I\ls L\ls I\ls P\ns J.\ns H\ls E\ls R\ls B\ls E\ls R\ls T}
%	\affiliation{Mathematics Institute, Zeeman Building, University of Warwick, Coventry. CV4 7AL. UK\\
%	emails\textup{\nocorr: \texttt{c.m.elliott@warwick.ac.uk}; \texttt{p.j.herbert@warwick.ac.uk}}}

\author{Charles M. Elliott\\
Mathematics Institute, Zeeman Building, University of Warwick,\\Coventry, CV4 7AL. UK\\
c.m.elliott@warwick.ac.uk\\
\\
Philip J. Herbert\\
Mathematics Institute, Zeeman Building, University of Warwick,\\Coventry, CV4 7AL. UK\\
p.j.herbert@warwick.ac.uk
}
\date{}

\maketitle
\begin{abstract}
	We consider a model of a biomembrane with attached proteins. The membrane is represented by a near spherical continuous surface and
	attached proteins are described as discrete rigid structures which attach to the membrane at a finite number of points. The resulting surface 
	 minimises a quadratic elastic energy (obtained by a perturbation of the Canham-Helfrich energy) subject to the point constraints which are imposed by the attachment of the proteins.  
	 We calculate the derivative of the  energy  with respect to protein configurations. The proteins are constrained to  
	 move tangentially by translation and by rotation  in the axis normal to a reference point. Previous studies have typically restricted themselves to a nearly flat membrane and circular inclusions.
A  numerically accessible representation of this derivative is derived and employed in some numerical experiments.


\end{abstract}
\section{Introduction}

The morphology of cell membranes and a variety of functions are well-known to be regulated 
by the interplay between surface proteins  and the curvature of the membrane. 
Biological membranes are composed of a lipid bilayer, this layer is believed to act like a 
fluid in the lateral direction and elastically in the normal direction.
This means that in principle, any proteins which may be embedded into or attached to the 
surface of the membrane may move freely.
This means that not only can the proteins influence the shape of the membrane, 
but also the protein interaction will be  membrane mediated.

Indeed, although direct protein-protein interactions are important,
%Many studies are devoted to direct protein-protein interactions \cite{?} however i
%It is important to study the membrane mediated interaction as
\cite{GouBruPin93} demonstrated that the long range interactions are predominantly membrane mediated.
An overview of membrane mediated interactions is given in \cite{BitConFou19}.
An assumption of symmetry of the protein inclusion allows for either analytic representation or approximation by an asymptotic expansion of the interactions \cite{KimNeuOst98,WeiKozHel98,DomFou02,YolHauDes14,FouGal15}.
Frequently the studies of these interactions were restricted to a nearly flat membrane with circular or single point inclusions.
It is known that the shape of the inclusion has a significant impact on the interaction \cite{KimNeuOst00}.
In the recent work of \cite{SchKoz15}, they consider a near spherical membrane which is deformed by particles which attach along segments of an ellipsoid or hyperbolid and in \cite{GraKie17}, they consider arbitrary, sufficiently regular, particle inclusions on a flat membrane.
Recent work has looked at shape formation of multiple smaller particles into larger structures \cite{Wei18,Gov18}.
The article \cite{DhaPer20} considers generic elastic energies on a manifold with embedded point particles which have a given interaction potential.
A variational formulation for equilibria of the surface and particle system is presented, along a discretisation.
Numerical validations are given, in particular, a Helfrich problem is presented.
We further note the work of \cite{ButNaz11} which considers point constraints in a Kirchoff plate, this bears a striking similarity to the biological problems of optimising the locations of constraints with respect to the an elastic membrane energy.

It is widely accepted in the literature that the near stationary 
state of lipid membranes are minimisers of the Canham-Helfrich energy \cite{Can70,Hel73},
\begin{equation}\label{eq:CH}
	\int_\mathcal{M} \left(\frac{\kappa}{2}(H-c_0)^2 + \sigma + \kappa_G K \right){\rm d} \mathcal{M}.
\end{equation}
Where the membrane is assumed to be thin and well modelled by a 2-dimensional surface $\mathcal{M}$, 
with the quantities $\kappa>0$, $\kappa_G\in \R$ are the bending rigidities associated to the mean and 
Gauss curvature respectively and $\sigma\geq 0$ is the surface tension.
For the principle curvatures of $\mathcal{M}$, $\kappa_1,\kappa_2$ we take $H:= \kappa_1+\kappa_2$ 
to be 2 times the usual value of the mean curvature and $K:= \kappa_1 \kappa_2$ the typical 
Gauss curvature.
The value $c_0\in \R$ is the spontaneous curvature, this corresponds to a mis-match between the inner and outer layers of the membrane, for example differing lipid composition.

We make some simplifying assumptions. The first is to set  $c_0 =0$, %this is justified by experiments \cite{some biology papers?} which demonstrate 
corresponding  to a physical assumption that the mismatch between the layers is rather small.
Another assumption is to neglect the  Gauss curvature term. This may  be justified by taking the rigidity $\kappa_G$ to be  constant and applying  the Gauss-Bonnet theorem, which states that 
when $\mathcal{M}$ is closed, the quantity $\int_\mathcal{M} K$ depends only on the Euler characteristic of $\mathcal{M}$. As we are considering a fixed topology of near-spherical membranes, we may ignore  this constant.
This leads to
\begin{equation}\label{eq:CHSimplified}
	\int_\mathcal{M} \left(\frac{\kappa}{2}H^2 + \sigma \right) {\rm d}\mathcal{M}.
\end{equation}

It is natural to introduce a volume constraint corresponding  to the membrane being 
impermeable and the fluid contained within the membrane being incompressible.
Indeed, without the volume constraint, it is known \eqref{eq:CHSimplified} is bounded below by $8\pi\kappa$ \cite{Wil65}
%Note on Embedded Surfaces, MonSch19}
and the degenerate sequence $\mathbb{S}^2\left(0,\frac{1}{n}\right)$ for $n\to \infty$ is a minimising sequence.
Further to this, we are interested in constraining $\mathcal{M}$ to contain a set of points, 
this corresponds to a protein in a fixed location being attached to the membrane.

We assume that the attached proteins are rigid, that is to say they do not bend and can 
only move by translations or rotations.
It is of clear interest to consider the force that the membrane exerts on these attached proteins.
This is relevant to, say, calculate locally minimising configuration of multiple proteins via a 
gradient flow, to estimate statistical quantities using over-damped Langevin 
Dynamics \cite[Section 2.2.2]{LelRouSto10} or as a step for a full model for the problem of particles in membranes.
For further details on estimation of the  free energy of a particle  membrane, see \cite{Kie19}.
%\marginpar{maybe some more interesting motivations?}
%One might also wish to use it to aid finding minimisers to a plate equation such as in \cite{ButNaz11}. \marginpar{would this actually be of interest to them?

The derivative of the energy with respect to particle location is calculated as a shape derivative in \cite{EllGraHob16}, and appears by use of a 
pull back method in \cite{GraKie17}, both in the case of large particles on a nearly flat membrane.
We will follow many of the ideas of this second work, making use of methods from \cite{ChuDjuEll19} 
to deal with the fact we are on a surface rather than a flat domain.
%The approach we use will not be limited to point particles, but t
%The typical model for bulk particles \cite{EllGraHob16,GraKie19} involves a constraint on the normal derivative at the boundary of the particle as well as the value at the boundary.
%We do not specifically consider these bulk particles.%, however we do define them and make clear where one would need something different for the theory to hold.
%We do not consider these bulk particles, however this would be an extension, which would require definitions which more closely match \cite{GraKie17}.\marginpar{Should I actually do the extension to the bulk particles?}
%We neglect this type of constraint due to the numerical scheme which we choose to use is necessarily able to handle a constraint on the normal derivative and 

%Although this work will focus on the case of a sphere and rotations, the theory developed may be applied to a generic surface with particle movements which correspond to translation and rotation in the tangent space.
\begin{comment}
	Our contribution to this subject is to extend this domain perturbation method for differentiability to a surface where we are also considering a non-local constraint (the volume constraint).
\end{comment}

One motivation for constructing a formula for the membrane mediated particle interactions may be seen from the following example.
For $\bar{\mathcal{E}}(p)$ the total energy of the particle system (the membrane energy with electrostatic interaction) in configuration $
p$, one might be interested in finding $p^*$ such that $\bar{\mathcal{E}}(p^*)$ is minimal.
One may choose to do this with a gradient descent algorithm
in which an update step might be:
\[
	p_{n+1} = p_{n} - \alpha_n \nabla_p \bar{\mathcal{E}} (p_n),
\]
for some $\alpha_n>0$ which may depend on $n$.
Clearly one may approximate the derivative $\nabla_p \bar{\mathcal{E}} (p_n)$ by taking a difference quotient.  
However this will be expensive, as  one would require solving $3 N +1$ linear systems - the system associated to the state $p_n$ and the $3 N$ directions that $\nabla_p$ corresponds to.
%The following section is devoted to the construction of $\nabla_p \mathcal{E}$ as a functional.
%This article constructs a functional which corresponds to the derivative.
With the explicit formula we find, the algorithm to construct the gradient would require solving $1$ linear system and evaluating $3 N$ functionals, where these functionals 
are relatively cheap to evaluate  compared to a linear solve for a fourth order PDE.

\subsection{Outline}
The quadratic energy
approximating  the general  Canham-Helfrich energy \eqref{eq:CHSimplified}
is presented in Section \ref{sec:Membrane+Particle} % by introducing %This is followed by a result on the regularity of the solution to the problem for a fixed particle location.
along with   precise definitions and notation for the attachments of  particles to the membrane.
The formula for the derivative of the minimising energy with respect to the location of the particles is    derived  in Section \ref{sec:CalculateDerivative}. 
%In Subsection \ref{subsec:ConstructDiffeo} we work to construct a diffeomorphism which (locally) reformulates the problem to be over a fixed function space and in Subsection\ref{sec:Reformulation} we give the reformulated energy% we reformulate the energy earlier introduced to be more simple to deal with by  pulling back the energy to be over a fixed function space, this will allows us to calculate a derivative.
Some numerical examples are presented in  Section \ref{sec:NumericalAnalysis}. 
In a finite element setting we calculate and compare   derivatives using the  formula and a difference quotient of the energies for comparison.
%We demonstrate that the energy depends on the particle configuration.

\subsection{Surface PDE preliminaries}
For completion, we now provide several definitions and results on the topic of surface PDEs 
which we will later need, the results may be found in \cite{DziEll13}.
For $\Gamma$ a closed sufficiently smooth  hypersurface in $\R^3$, 
there is a bounded domain $\Omega \subset \R^3$ such that $\Gamma = \partial \Omega$.
The unit normal to $\Gamma$, $\nu$, that points away from $\Omega$ is called the outwards unit normal.
Define $P_\Gamma:= \I - \nu \otimes \nu$ on $\Gamma$ to be, at each point $x \in \Gamma$, the projection onto 
the tangent space at that point, $T_x\Gamma$, where $\I$ is the identity matrix and for $a,b\in \R^n$, $a\otimes b:=a b^T$.
% with $(a\otimes b)_{ij}= a_ib_j$.
For a differentiable function $f$ on $\Gamma$, we define the tangential gradient
\[
	\nabla_\Gamma f:= P_\Gamma \nabla \tilde{f},
\]
where $\tilde{f}$ is a differentiable extension of $f$ to an open neighbourhood of $\Gamma$ in $\R^3$.
Here $\nabla$ is the standard derivative on $\R^3$.
Lemma 2.4 of \cite{DziEll13} shows this definition is independent of the choice of extension $\tilde{f}$.
The components of the tangential gradient are denoted
\[
	\underline{D}_i f := (\nabla_\Gamma f)_i.
\]
The map $\mathcal{H}:=\nabla_\Gamma \nu$ is called the extended Weingarten map and is symmetric with 
zero eigenvalue in the normal direction.
The mean curvature $H$ is given as the trace of $\mathcal{H}$.
%The eigenvalues $\kappa_1,\,\kappa_2$ associated to the tangential eigenvectors are the principal curvatures of $\Gamma$.
For a twice differentiable function, the Laplace-Beltrami operator is defined to be
\[
	\Delta_\Gamma f := \nabla_\Gamma \cdot \nabla_\Gamma f = \sum_{i=1}^{3}\underline{D}_i \underline{D}_i f.
\]
We write $D_\Gamma^2 f$ to be the surface Hessian and Lemma 2.6 in 
\cite{DziEll13} shows that the surface Hessian is, in general, not symmetric with the relation
\begin{equation}\label{eq:SurfaceCommutate}
	\underline{D}_i \underline{D}_j f - \underline{D}_j \underline{D}_i f
	=
	(\mathcal{H}\nabla_\Gamma f)_j \nu_i
	-
	(\mathcal{H}\nabla_\Gamma f)_i \nu_j.
\end{equation}
It is well-known \cite[Lemma 2.8]{DziEll13} that there is a small neighbourhood 
around $\Gamma$ of 
width $\delta>0$, $\mathcal{N}_\delta$, and 
maps $d\colon \mathcal{N}_\delta \to \R$, the oriented distance function, and $\pi\colon \mathcal{N}_\delta\to \Gamma$, the closest point projection,
such that for any $\tilde{X} \in \mathcal{N}_\delta$ we may uniquely decompose 
\begin{equation}\label{eq:decompose}
	\tilde{X} = \pi(\tilde{X}) + d(\tilde{X})\nu(\pi(\tilde{X})).
\end{equation}

\section{Membrane and particle model}\label{sec:Membrane+Particle}
%\marginpar{need some more definitions here}
We begin with the deformation model for  the membrane along with model for the 
particles and their attachment  to the membrane.

\subsection{Membrane model}
We now fix $\Gamma:= \mathbb{S}^2(0,R)$ to be  the 2-sphere of radius $R$, for a given $R>0$.
In light of this, we see that for $X \in \R^3 \setminus\{0\}$, $\pi(X)= R \frac{X}{|X|}$ and $d(X) = |X|-R$.
We are interested in finding a surface which is  a near spherical membrane of the form  
\[
	\mathcal{M}(v):= \left\{ x+ \rho v(x) \nu(x) : x \in \Gamma\right\}.
\]
 where $\rho$ is small and $v$ is  sufficiently smooth. Thus  $\mathcal M(v)$ is a graph over $\Gamma$. 
We use the following energy:
\begin{equation}\label{eq:QuadraticSurfaceEnergy}
	%J(\mathcal{M}(u)) = 
	J(v) := \frac{1}{2}\int_{\Gamma} \kappa(\Delta_\Gamma v)^2 + \left( \sigma - \frac{2\kappa}{R^2} \right)|\nabla_\Gamma v|^2 - \frac{2\sigma}{R^2}v^2
\end{equation}
derived by \cite{EllFriHob17}.
It is seen for $\int_\Gamma v = 0$ that $J(v)$ is the first non-trivial term of the Taylor expansion in $\rho$ of the Lagrangian induced by the Canham-Helfrich energy for surfaces with enclosed volume constrained to be $\frac{4}{3}\pi R^3$ around the critical point $(\Gamma,-\frac{\sigma}{R})$.
\begin{comment}
It is shown for sufficiently small deformations that 
\eqref{eq:QuadraticSurfaceEnergy} approximates %approximation of 
the Lagrangian induced by the Canham-Helfrich energy \eqref{eq:CHSimplified} with constrained enclosed volume.  %, when considering the surface to be sufficiently close to a sphere of radius $R$.
%As such, the minimum of this approximates the stationary state of the membrane.
%, it could be seen as a relatively easy extension to handle this change.
\end{comment}
This energy is analogous to the Monge-Gauge for a nearly flat membrane, \cite{EllGraHob16}, which is   formally obtained by  taking the limit $R \to \infty$.

\begin{definition}
We define a bilinear form  $a\colon H^2(\Gamma) \times H^2(\Gamma)\to \R$ to be
	\begin{equation}\label{eq:BilinearForm}
	a(\eta,v):= \int_\Gamma\kappa  \Delta_\Gamma \eta \Delta_\Gamma v + \left(\sigma - \frac{2\kappa}{R^2}\right)\nabla_\Gamma \eta \cdot \nabla_\Gamma v - \frac{2\sigma}{R^2}\eta v ~~ \forall \eta,v \in H^2(\Gamma),
	\end{equation}
	which is the bilinear form given by the first variation of \eqref{eq:QuadraticSurfaceEnergy}.
	We define the space
	\[
		U := \left\{ v \in H^2(\Gamma) : \int_\Gamma v = 0\right\}.
	\]
\end{definition}
\begin{remark}
We note that under the small deformation methodology of {\cite{EllFriHob17}} that one may deal with appropriately small $c_0$ as considered in {\cite{EllHatHer20,EllHat19}}.
\end{remark}

\subsection{An energy  minimising membrane subject to point constraints}
With the above definitions, one may now write the following problem:
\begin{problem}\label{prob:SingleParticleEnergyMinimisation}
Given $Z \in \R^K$ and $\mathcal C=\{X_j\in \Gamma,\, j=1,...,K\}$, find $u \in U$ such that $J(u)$ is minimised subject to $u(X_j) = Z_j$ for $j=1,...,K$.
\end{problem}

This defines $K$  point constraints on $u$  and  is admissible for $u\in H^2(\Gamma)$   because of  the well known 
embedding for 2 dimensions, $H^2(\Gamma)\subset C(\Gamma)$ \cite{AdaFou03}.
\begin{comment}
\begin{remark}
In the case of a bulk particle, we assume that the particle attaches at a fixed curve, which we again write as $\mathcal{G}$, we also require that there is a unit normal field, $N\colon \mathcal{G}\to \R^3$.
The inclusion of the particle within the membrane may then be written as
\[
	\mathcal{G}\subset \mathcal{M}(u), ~~ N(x) \in  T_x \mathcal{M}(u) ~ \forall x \in \mathcal{G}.
\]
For $\tilde{X} = \pi(\tilde{X}) + d(\tilde{X})\nu(\pi(\tilde{X}))$, this first inclusion becomes $u(\pi(\tilde{X})) = d(\tilde{X})$, exactly as for the point constraints, with the second constraint becoming $\partial_{\tilde{\nu}} u = \frac{\nu\cdot N}{\tilde{\nu}\cdot N}$, where $\tilde{\nu}$ is the unit co-normal to $\pi(\mathcal{G})$.%$\partial_{\tilde{\nu}} u = ((I-P_\Gamma )N)\cdot \tilde{\nu}$.
\marginpar{Unsure how to correctly define this constraint, seems reasonable,  probably needs justification - corresponds to $\tan$(angle) between the co-normal and normal on $\Gamma$, could just remove}
\end{remark}
\end{comment}

We have the following well-posedness and regularity result.
The well-posedness follows from \cite[Theorem 5.1]{EllHer20-A} while the regularity result may be found in Appendix \ref{app:EllReg}.
\begin{theorem}\label{thm:ExistenceRegularity}
	Suppose $K\geq 4$ and the points of $\mathcal{C}$ do not lie in a single plane.
	Then there there is a unique $u\in U$ which solves Problem \ref{prob:SingleParticleEnergyMinimisation}.
	Furthermore, for any $p \in (1,2)$, it holds that $u \in W^{3,p}(\Gamma)$.
\end{theorem}
\begin{remark}\hspace{0.1cm}
\begin{itemize}
\item
The fact the solution of Problem \ref{prob:SingleParticleEnergyMinimisation} has three weak derivatives will be used to give a more convenient representation of the derivative we calculate.
\item
A related  problem has been considered in {\cite{EllFriHob17}},
%  including sufficient conditions on the location 
%and number of point constraints for uniqueness of a minimiser when only having a constraint on the 
%mean value, rather than also constraining the center of mass, which is seen to also be a sufficient condition.
where the authors consider the minimisation over a smaller space which enforces a fixed centre of mass for the membrane.
\item
The works {\cite{EllGraHob16,GraKie17,GraKie19}} consider a larger solution space whereby the particles may, in some sense, tilt.
The problem for this tilting on a sphere, or general domain, is of interest and may be studied in future work.
%by assuming any of the 'out of plane' movements are sufficiently small and using an appropriate linearisation.
\item
An example of non-uniqueness for $K>4$ would be to consider $\mathcal{C}\subset \{ x \in \Gamma : x_1 = 0\}$.
Then for $u$ a solution of Problem \ref{prob:SingleParticleEnergyMinimisation}, we see that $u + \alpha \nu_1 \in U$ and $J(u + \alpha\nu_1) = J(u)$ for any $\alpha \in \R$.
\end{itemize}
\end{remark}

\subsection{A single particle model}
 We wish to model the attachment of proteins  to a biomembrane. A protein is considered to be a rigid discrete
 structure which is attached to the membrane at a finite number of fixed points. An example would be  a protein such as 
 FCHo2 F-BAR domains, where it is understood that a small number of atoms are more likely to attach to the 
 membrane \cite{HenBouMei10,HenKenFor07}. This is in contrast to the case mainly considered in \cite{GraKie19}, 
 where the protein is modelled as being embedded in the membrane and attached along a curved boundary. The protein biomembrane interaction is modelled by attachment at these points.
To begin, we restrict ourselves to a single protein in order to establish notation.
We describe the protein by a finite set of distinct points $\mathcal{G}:=\{\tilde X_i \in  \R^3,\, i=1,...,M\}$. 
The points of $\mathcal{G}$ correspond to charged ends of the protein which attach to the membrane.
The attachment constraint is the requirement that $\mathcal{G}$ is contained in the graph $\mathcal{M}(u)$ which we write as 
\begin{equation}\label{eq:GeometricInclusion}
	\mathcal{G}\subset \mathcal{M}(u).
\end{equation}
%:=\{X \in \mathbb R^3~|~X=x+u(x)\nu(x), x\in \Gamma\}$$
  %over $\Gamma$, where $u$ is sufficiently smooth and $\nu(X)$ denotes the outward pointing unit  normal to $\Gamma$ at $x$. This may be %, this is clear as we are interested in surfaces which are a graph over $\Gamma$.
%The points which we will consider being embedded in the membrane may be described by a finite set of discrete points $\mathcal{G}\subset \R^3$ which represents a protein, or more specifically the charged ends of a protein.
%We will consider the constraint that
% the charged ends of the protein
%Clearly, the condition \eqref{eq:GeometricInclusion} implies a restriction that $\mathcal{G}$ must be a graph over $\Gamma$.  %, otherwise one will be seeking a minimiser in the empty set.
% the model that we are looking for a minimiser which is a graph over $\Gamma$ breaks down.
%As the energy \eqref{eq:QuadraticSurfaceEnergy} is only valid for sufficiently small $u$, this means that $\mathcal{G}$ must be appropriately close to $\Gamma$.
%In light of this, we will assume $\mathcal{G} \subset \mathcal{N}_\delta$, where $\mathcal{N}_\delta$ is a $\delta$-width neighbourhood of $\Gamma$ where \eqref{eq:decompose} is well defined.
It follows that
%, for $\mathcal{G}\subset \mathcal{N}_\delta$,
any $\tilde{X} \in \mathcal{G}$ may be uniquely decomposed into
\[
	\tilde{X} = \pi(\tilde{X}) + d(\tilde{X}) \nu(\pi(\tilde{X})) = R \frac{\tilde{X}}{|\tilde{X}|} + \left(|\tilde{X}|-R\right)\frac{\tilde{X}}{|\tilde{X}|}
\]
and the condition \eqref{eq:GeometricInclusion} becomes
\begin{equation}\label{eq:geometricConstraint2}
	u(\pi(\tilde{X})) = d(\tilde{X}) ~~ \forall \tilde{X} \in \mathcal{G}.
\end{equation}
For ease of notation, we write $X:= \pi(\tilde{X})$, $z:=d(\tilde{X})$ and index the points of $\mathcal{G}$ so that $\{\tilde{X}_i\}_{i=1}^M = \mathcal{G}$,
hence we may write \eqref{eq:geometricConstraint2} as
\begin{equation}\label{eq:geometricConstraint3}
	u(X_i) = z_i~~\forall i=1,...,M.
\end{equation}
\begin{definition}\label{def:SitesOfAttach}%\marginpar{probably needs to be changed to be a stand alone definition}
We write $\cee:= \{ \pi(\tilde{X}):\tilde{X} \in \mathcal{G}\}= \{X_i\}_{i=1}^M $ to be the sites of attachment.
%\{ x \in \Gamma : x = \pi(\tilde{X}) ~\mbox{for some}~ \tilde{X} \in \mathcal{G}\}
%With the idea of multiple proteins being considered,
Furthermore, we write
\[
	u|_{\cee} = Z
\]
to be shorthand for \eqref{eq:geometricConstraint3}.
\end{definition}

%as the natural space to minimise \eqref{eq:QuadraticSurfaceEnergy} is a subset of $H^2(\Gamma)$

\subsection{Parametrisation of a single particle}\label{subsec:ParaSingleParticle}
%\marginpar{Subsection $\to$ Section?}
%\begin{remark}
We now parameterise the movement of a single particle.
We attempt to keep our notation as similar as possible to that of \cite{GraKie17} which deals with the movement of curves in a flat domain, in contrast to our points which move on a sphere.

The assumption  that the protein is rigid  is meant in the sense that any 
movement of $\mathcal{G}$ should preserve the orientation and the distance between points.
There are 6 degrees of freedom by which $\mathcal{G}$ can be moved, this is translation and rotation.
We further restrict to lateral (i.e. tangential) movement of $\mathcal{G}$ over the membrane.
This means that the height of attachment above $\Gamma$, the values $Z$, will be independent of any movement.
In the flat setting these lateral movements correspond to rotation perpendicular to the plane and 
translation within the plane.
Although this is a strong restriction to make to the full model, it is important in this setting to avoid the particle moving out of the graph-like description.
%However in this setting one must avoid the particle 
%moving out of the graph-like description.
%\marginpar{picture of the breakdown? $\cdot~ \cdot \to :$}%\end{remark}

The configuration of a single   particle $\mathcal G$ is defined by a rigid transformation from a fixed position.  
We associate one point  $X_\mathcal G\in \Gamma$ with $\mathcal G$. We call $X_\mathcal G$   the {\it centre} of $\mathcal G$. 
The configuration of the particle is defined by a rotation about the axis defined by $\nu(X_\mathcal{G})$ 
together with a \emph{tangential translation} of $X_\mathcal{G}$ along the surface of $\Gamma$.
A rotation around $\nu(X_\mathcal G)$ is characterised by an angle, $\alpha\in \mathbb R$.
A tangential translation is characterised by a tangent vector $\tau \in T_{X_\mathcal{G}} \Gamma\cong \R^2$.
For this tangent vector, the idea is to consider the transport of $X_\mathcal{G}$ along the geodesic defined by $\tau$ and that the other points should follow with a rigid transformation.
In the setting of a sphere, this corresponds to rotating the points by angle $|\tau|$ in the axis perpendicular to both $\nu(X_\mathcal{G})$ and $\tau$.
%and   rotating  by angle $|\tau|$ in the plane spanned by $\nu(X_{\mathcal{G}})$ and $\tau$.
Thus for a particle with 
centre $X_\mathcal G$ we write $\mathcal G(p)$, $p=(\alpha,\tau)$ to be as described above, leading to the following  definition of particle configuration.

\begin{definition}\label{def:MoveSingleParticle}
Given particle $\mathcal{G}\subset \R^3$ with centre $X_\mathcal G$ and $p = (\alpha,\tau) \in \R \times T_{X_\mathcal{G}} \Gamma$, we write
\[
	\mathcal{G}(p) := \{\phi(p,\tilde{X}):\tilde{X} \in \mathcal{G} \},%\phi(\mathcal{G},p),
	%\{ x \in \R^3 : x = \phi(p,\tilde{X}) ~\mbox{for some}~\tilde{X} \in \mathcal{G} \}
\]
with
\begin{equation}\label{eq:RigidBodyMotion}
	\phi(p,x):=R_T(\tau) R_n(\alpha)x~~ \forall x \in \R^3,
\end{equation}
where $R_n(\alpha)$ is given by
\[
	R_n(\alpha)x := (\nu(X_\mathcal{G}) \otimes \nu(X_\mathcal{G})) x 
	+ \cos(\alpha) (\nu(X_\mathcal{G}) \times x) \times \nu(X_\mathcal{G}) 
	+ \sin(\alpha) (\nu(X_\mathcal{G}) \times x),
\]
and for $\tau \neq 0$, define $\tilde{\tau}:= \nu(X_\mathcal{G})\times\frac{\tau}{|\tau|}$, $R_T(\tau)$ is given by
\[
	R_T(\tau) x := (\tilde{\tau}\otimes \tilde{\tau})x
	+
	\cos(|{\tau}|) (\tilde{\tau}\times x)\times \tilde{\tau}
	+ \sin(|{\tau}|)(\tilde{\tau}\times x),
\]
and $R_T(0) x = x$.
A diagram showing the transformations $R_n$ and $R_T$ may be found in Figure \ref{fig:DiagramRotations}.
%
%
%where $R_{i}$ are the rotations in the $x_i$-axes, that is,
%\[
%	R_i(\alpha)x:= \left(e_i \otimes e_i\right)x + \cos(\alpha)\left(e_i \times x \right)\times e_i + \sin(\alpha)\left(e_i \times x\right).
%\]
Furthermore, write
\[
	\cee(p):= \{\phi(p,X): X \in \cee\}, %\phi(B^0,p) = \pi(\mathcal{G}(p)),
	%\{x\in\Gamma : x = \phi(p,X) ~\mbox{for some}~ X \in \cee\}
\]
this coincides with the projection of $\mathcal{G}(p)$ onto $\Gamma$.
\end{definition}
\begin{figure}\begin{center}
\subfigure[Diagram for $R_n$.]{\includegraphics[scale=0.07125]{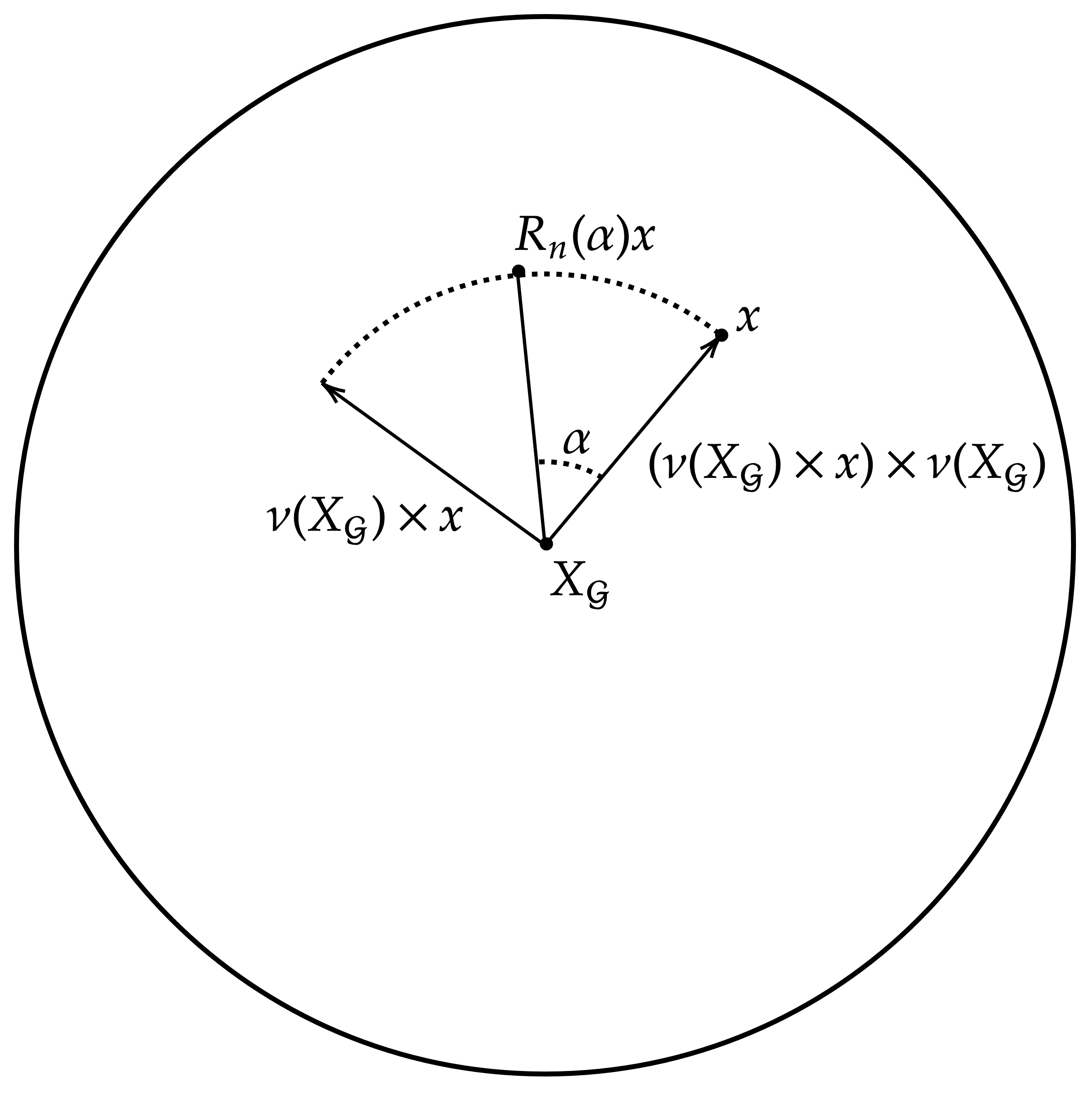}}
\subfigure[Diagram for $R_T$.%, the dotted line represents the points $y$ such that $y \cdot \tilde{\tau}= \frac{x}{|x|}\cdot \tilde{\tau}$.
]{\includegraphics[scale=0.07125]{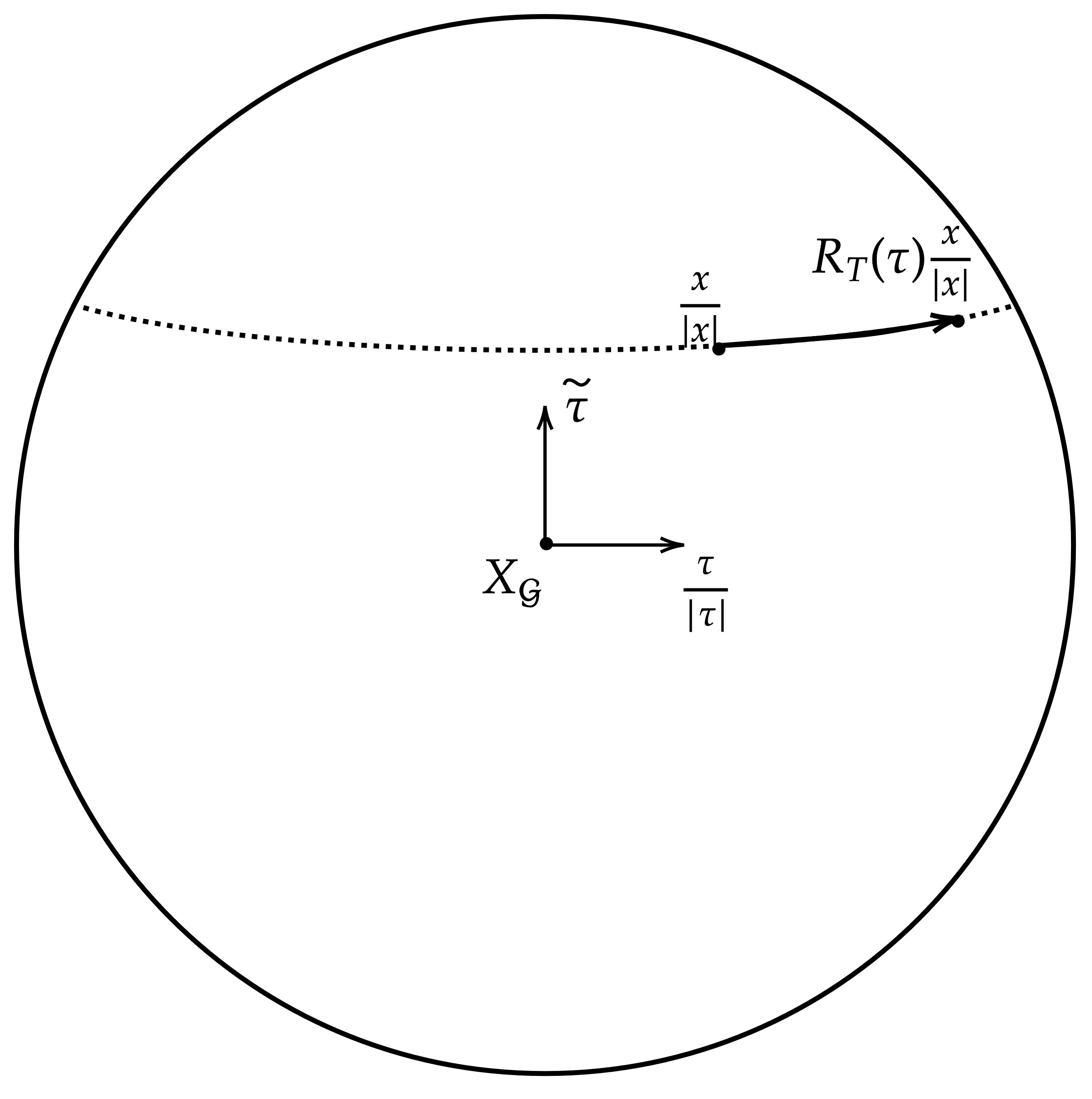}}
\caption{Diagrams demonstrating the transformations $R_n$ and $R_T$, both with $\nu(X_\mathcal{G})$ coming out of the page.}
\label{fig:DiagramRotations}
\end{center}
\end{figure}
\begin{remark}
	The choice that $\phi(p,x):= R_T(\tau)R_n(\alpha)x$ rather than $R_n(\alpha)R_T(\tau)x$ is arbitrary.
	It is clear that they will both generate the same family of configurations.
\end{remark}
We notice $\mathcal G=\mathcal{G}(0) $ and similarly $\cee= \cee(0)$.
We further note that $p$ is periodic in the following sense.
For $p = (\alpha,\tau)$, $\bar{p} = (\alpha+ 2\pi,\tau)$ and $\tilde{p}= (\alpha,\tau + 2\pi \frac{\tau}{|\tau|})$ it holds,
\[
	\phi(p,\cdot) \equiv \phi(\bar{p},\cdot ) \equiv \phi(\tilde{p},\cdot).
\]
Further note that if $\mathcal{G}$ contains only one point, $\tilde{X}_1$, and one sets $X_\mathcal{G}= X_1$, it is seen that $\alpha$ becomes a redundant parameter.

\subsection{Configuration of particles}
We now make the extension to multiple groups of particles.%\marginpar{might want to discuss how this can be generalised for generic surfaces}

\begin{definition}\label{def:MultiplePointsMove}
Given discrete sets with finite number of points,
\[
	\mathcal{G}_1,...,\mathcal{G}_N \subset \mathcal{N}_\delta \subset \R^3,
\]
%sufficiently close to $\Gamma$,
we write
\[
	\cee_i := \{\pi(\tilde{X}):\tilde{X}\in \mathcal{G}_i \}  ~\mbox{for}~ i=1,...,N,%\pi(\mathcal{G}_i)
	%\{x \in \Gamma : x=\pi(\tilde{X}) ~\mbox{for some}~ \tilde{X}\in \mathcal{G}_i \}  ~\mbox{for}~ i=1,...,N,
\]
the projection of $\mathcal{G}_i$ onto $\Gamma$.
Let the $\mathcal{G}_{1},...,\mathcal{G}_N$ have centres $X_{\mathcal{G}_1},...,X_{\mathcal{G}_N}$  and let $p= (p_1,...,p_N) \in \prod_{i=1}^N (\R \times T_{X_{\mathcal{G}_i}} \Gamma)$, 
 where  $p_i=(\alpha_i,\tau_i) \in \R \times T_{X_{\mathcal{G}_i}} \Gamma$ we define
\[
	\phi_i(p,x) := R_{T_i}(\tau_i)R_{n_i}(\alpha_i)x~~ \forall x \in \R^3,
\]
where the operators $R_{T_i}(\tau_i),\, R_{n_i}(\alpha_i)$ are defined relative to the centres $X_{\mathcal{G}_i}$, as in Definition \ref{def:MoveSingleParticle}.

Further define
\[
	\mathcal{G}_i(p) := \{\phi_i(p,\tilde{X}): \tilde{X}\in \mathcal{G}_i \}~\mbox{for}~ i=1,...,N,%\phi_i(\mathcal{G}_i,p_i)
	%\{x \in \R^3 : x=\phi_i(\tilde{X},p_i) ~\mbox{for some}~ \tilde{X}\in \mathcal{G}_i \}~\mbox{for}~ i=1,...,N,
\]
and
\[
	\cee_i(p):= \left\{\phi_i(p,X) : X\in \cee_i \right\}  ~\mbox{for}~ i=1,...,N,%\pi(\mathcal{G}_i(p))
	%\left\{x \in \Gamma : x=\phi(X,p_i) ~\mbox{for some}~ X\in \cee_i \right\}  ~\mbox{for}~ i=1,...,N,
\]
the projection of $\mathcal{G}_i(p)$ onto $\Gamma$. Observe that
\[ \mathcal G_i(0)=\mathcal G_i~~\mbox{and}~~~\mathcal C_i(0)=\mathcal C_i,~~~i=1,...,N.
\]

\begin{comment}
\[
	R_{n_i}(\alpha_i)x := (\nu(X_{\mathcal{G}_i}) \otimes \nu(X_{\mathcal{G}_i})) x 
	+ \cos(\alpha_i) (\nu(X_{\mathcal{G}_i}) \times x) \times \nu(X_{\mathcal{G}_i}) 
	+ \sin(\alpha_i) (\nu(X_{\mathcal{G}_i}) \times x),
\]
and for $\tilde{\tau}_i:= \tau_i \times \nu(X_{\mathcal{G}_i})$, $R_{T_i}$ is given by
\[
	R_{T_i} x := (\tilde{\tau}_i\otimes \tilde{\tau}_i)x
	+
	\cos(|{\tau_i}|) (\tilde{\tau}_i\times x)\times \tilde{\tau}_i
	+ \sin(|{\tau_i}|)(\tilde{\tau}_i\times x).
\]
\end{comment}

%, where the $\mathcal{G}_i(p)$ is taken as in Definition \ref{def:def:MoveSingleParticle}, with $\phi_i$ being as in \eqref{eq:RigidBodyMotion} relative to each $X_0^i$.
\end{definition}

\begin{definition}
	We define the set of feasible particle configurations to be
	\[
		\Lambda^{\circ} := \left\{p \in \prod_{i=1}^N (\R \times T_{X_{\mathcal{G}_i}} \Gamma ): \forall i,\,j=1,...,N, ~i\neq j,~ \cee_i(p) \cap \cee_j(p) = \emptyset \right\}.
	\]
%	where $p_k\in \R^3$ is the $k$-th column of $p$.
	We define the closure of the set of feasible particle configuration by $\Lambda := \overline{\Lambda^\circ}$.
	Furthermore, for $p \in \Lambda$ we define
	\[
		\Gamma(p):= \Gamma \setminus \bigcup_{i=1}^N \cee_i(p).
	\]
\end{definition}
We first note that $0\in \prod_{i=1}^N(\R\times T_{X_{\mathcal{G}_i}})$ is not a distinguished configuration. 
% (indeed it need not be in the set $\Lambda^\circ$).
Given any  non-overlapping initial configuration of particles $\{\mathcal C_i\}_{i=1}^N$, it is clear that $\Lambda^\circ$ is the set of all possible configurations of 
particles which have been moved by the rigid motions parametrised by $p$ described at the start of Section \ref{subsec:ParaSingleParticle}.
\begin{figure}\begin{center}
\includegraphics[width=0.9\linewidth]{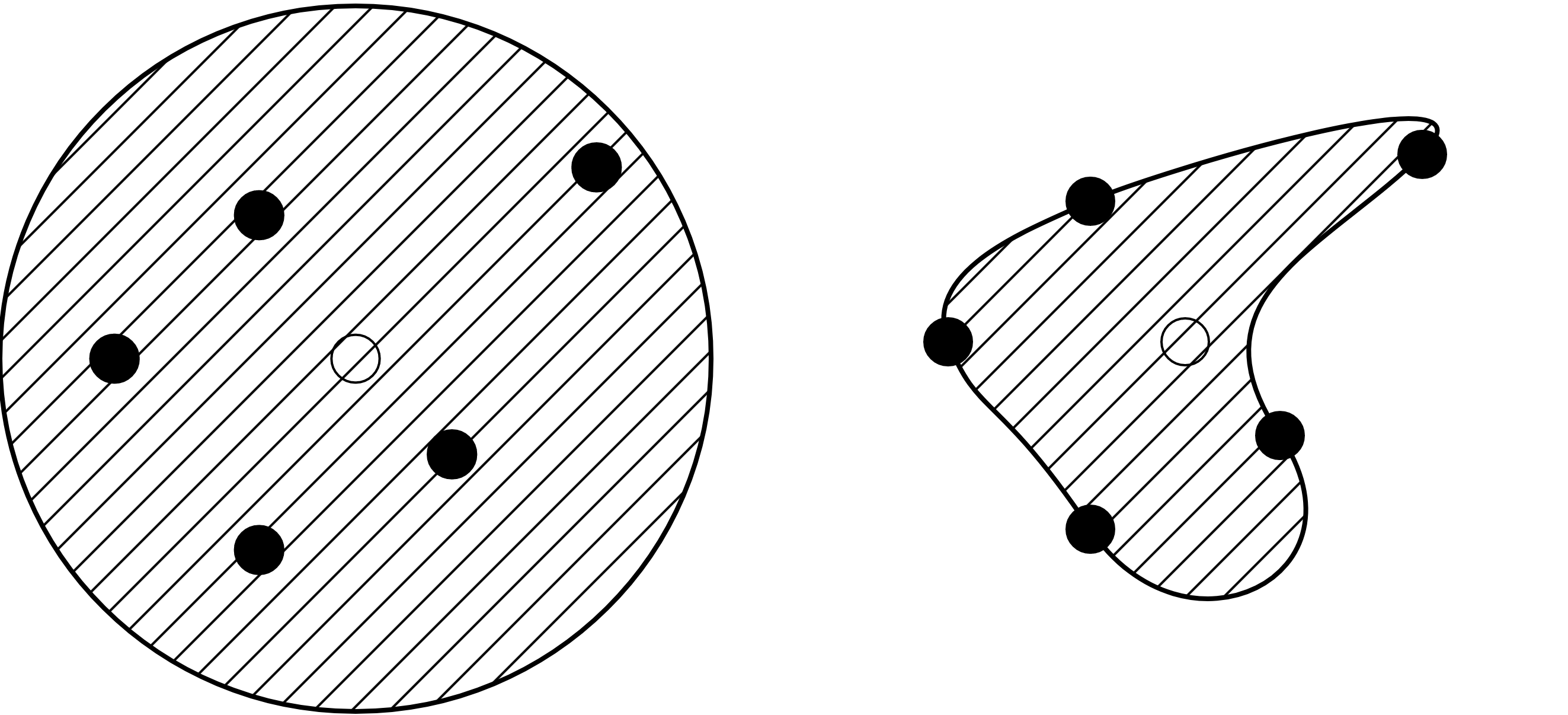}
\caption{Diagram indicating the different areas which might be excluded from having part of another particle in for two idential particles.
The left is the radius approach, on the right the area is given by the interior of a curve passing through all the points.}\label{fig:ExclusionAreas}
\end{center}
\end{figure} 

%\subsection{Notation for particles}
\begin{remark}
Notice that for $p \in \Lambda^\circ$, it may hold that the 'interiors' of particles overlap.
As such one might want to consider a subset of $\Lambda^\circ$ whereby one defines an appropriate interior of particles and assumes that the intersection of these is empty, or perhaps one may also assign a 'radius' to each particle and consider the set where there are no points from another particle which lie inside this radius.
Two ideas of these exclusion areas are shown in Figure \ref{fig:ExclusionAreas}.
In this diagram, the clear dot is the centre of a particle and the black dots are the points of the particle and the exclusion area is signified by the hatched lines.
The choice of this subset is not of importance when constructing the derivative, but is important when considering which particle configurations are admissible.
Requiring that the particles do not overlap could be included as part of a Lennard-Jones potential, see (4.18) in \cite{EllGraHob16}, where it could be seen that this discussion pertains to a choice of the distance function in their formula.
\end{remark}

% before Definition \ref{def:MoveSingleParticle}.

For each $p\in \Lambda^\circ$ we have a set of point constraints on elements of $H^2(\Gamma)$. This motivates the following parameterised trace operators.
\begin{definition}\label{def:DefineT} Given $p \in \Lambda^\circ$:-
\begin{itemize}
\item
	 For $i=1,...,N$, define the maps $T_i(p)\colon H^2(\Gamma) \to \R^{|\cee_i|}$ by
	\[
		T_i(p) \colon v \mapsto \left( v\circ \phi_i(p,\cdot)\right)|_{\cee_i},
	\]
	where $\phi_i(p,\cdot)|_{\cee_i}$ is meant as in Definition \ref{def:SitesOfAttach}.
	\item
	For $v \in H^2(\Gamma),$ $Z\in \prod_{i=1}^N \R^{|\cee_i|},$ we say $T(p)v = Z$ when
	\[
		T_i(p) v = Z_i\in \R^{|\cee_i|} ~\mbox{for} ~ i=1,...,N
	\]
	where $Z$ is given by the particles $\mathcal{G}_1,...,\mathcal{G}_N$.
	\item Define the following  subsets of $H^2(\Gamma)$
	\[
		U(p):= \left\{ v \in U :  T(p) v =Z\right\},
	\]
	\[
		U_0(p):= \left\{ v \in U : T(p) v = 0\right\}.
	\]
\end{itemize}
\end{definition}

\begin{assumption}Henceforth, we assume that there is $l$, $1\leq l \leq N$, such that $\cee_l$ is not coplanar.
\end{assumption}%this holds.% satisfying this condition.% over $\Lambda^\circ$.%, else restrict $\Lambda^\circ$ to a slightly smaller set where this holds.
\begin{definition}[Membrane configurational energy]\label{def:solution}
	Given $p\in \Lambda^\circ$, we define $u(p) \in U(p)$ by
	\[
		u(p) := \argmin_{v \in U(p)} J(v)
	\]
	and we define the membrane configurational energy $\mathcal{E}\colon \Lambda^\circ \to \R$ by
	\[
		\mathcal{E}(p) := J(u(p)).
	\]
%	For a given interaction potential, $I(p):=\sum_{i,j=1,\, i\neq j}^N \sum_{X(p_i) \in \mathcal{G}_i(p)}\sum_{Y(p_j) \in \mathcal{G}_j(p_j)}$**summand** ,
%	we define the configurational energy
%	\[
%		\bar{\mathcal{E}}(p) = \mathcal{E}(p)+I(p).
%	\]
\end{definition}

It is clear that, by a trivial extension to Theorem \ref{thm:ExistenceRegularity}, $u(p)$ exists, is unique and satisfies $u(p) \in W^{3,2-\delta}(\Gamma)$ for any $\delta \in (0,1)$.
For $p \in \partial\Lambda^\circ$ we do not necessarily have that a $u(p)$ exists, this is due to $U(p)$ possibly being empty.
% this problem can be relaxed by the use of so called 'soft constraints' which are discussed in \cite{EllGraHob16,GraKie19,EllHatHer20,EllHer19}.% and correspond to penalising the distance between the particle and membrane with an elastic type energy.

\begin{remark}
	Notice that $\mathcal{E}$ may not be the total energy associated to the particle-membrane configuration. For example,  $\mathcal{E}$ may be augmented with a pairwise interaction 
% energy of the form
%	\[
%		\sum_{i,j=1,\, i\neq j}^N \sum_{X(p_i) \in \mathcal{G}_i(p)}\sum_{Y(p_j) \in \mathcal{G}_j(p_j)} F(X(p_i),Y(p_j)),
%	\]
between particles modeling  forces between different particles.%, where $F$ is a given pairwise potential.
\end{remark}

\section{Gradient of the energy  with respect to configuration changes}\label{sec:CalculateDerivative}
In this section  we find a formula for  the derivative of $\mathcal{E}(p)$ with with respect to 
changes in the configuration $p$.
\begin{definition}[Derivative of the configurational energy]
	The configurational energy is differentiable at $p\in \Lambda^\circ$ in the direction   
	$e\in\prod_{i=1}^N (\R \times T_{X_{\mathcal{G}_i}}\Gamma)$ if the derivative 
	\[
		\frac{\dee}{\dee t}\mathcal{E}(p+te)|_{t=0},
	\]
	 exists and we denote this by
	 $\partial_e \mathcal{E}(p)$.
\end{definition}

The difficulty lies in the implicit definition of the energy 
$\mathcal{E}(p)$ in terms of the minimisation of the  quadratic energy $J(v)$ over the configurational space $U(p)$ requiring the evaluation of
$$%\frac{\dee}{\dee t}\mathcal{E}(p+te)|_{t=0}=
\frac{\dee}{\dee t}J(u(p+te))|_{t=0}$$
which involves the minimisation of $J(\cdot)$ over $U(p+te)$.  In order to achieve this we fix $p$ and  employ suitable 
 local isomorphisms on the vector spaces $U(p)$ via appropriate diffeomorphisms of the domain $\Gamma(p)$.
%employing  a domain perturbation approach. 
This is applied locally to transform the 
energy \eqref{eq:QuadraticSurfaceEnergy} and the related minimisation problems over a reference function space. 

We make the following assumption:
%For the moment, we make the following assumption.
%We will show this assumption is satisfied in Subsection \ref{subsec:ConstructDiffeo}.
\begin{assumption}\label{ass:ThereExistsDiffeo}
	Let $k\geq 3$.  For each $p \in \Lambda^\circ$ there exists an open ball $\mathcal{B}\subset \prod_{i=1}^N \left( \R \times T_{X_{\mathcal{G}_i}} \Gamma\right)$ containing $0$  and a family of $C^k$-diffeomorphisms $\chi\colon \mathcal{B}\times\Gamma\to\Gamma$ such that
	$$\chi(0,\cdot) \mbox{ is the identity on}~ \Gamma$$
and	for all $q \in \mathcal{B}$, $p+q \in \Lambda^\circ$ and   
	 
	\begin{equation}\label{eq:ConditionsForDiffeo}
		 v\circ\chi(q,\cdot)^{-1} \in U(p+q) \iff v \in {U}(p).%~~~~ \forall q \in \mathcal{B}, p+q \in \Lambda^\circ.
	\end{equation}

\end{assumption}
	
We now define what we mean by the derivative of $\chi$ with respect to $e$.
	\begin{definition}
		Given $q \in \mathcal{B}$ and $e \in \prod_{i=1}^N \left( \R \times T_{X_{\mathcal{G}_i}} \Gamma\right)$, for each $x \in \Gamma$, the derivative of $\chi(\cdot,x)$ at $q$ in direction $e$ is defined to be
		\[
			\partial_e \chi(q,x):= \frac{\dee}{\dee t} \chi(q+t e, x) |_{t=0}.
		\]
\end{definition}

	\begin{remark}
	
 Notice that:
 \begin{itemize}
 \item
 The dependence on $p$ of  $\mathcal B$ and $\chi$ has been suppressed.
 \item
 For our purposes we will not require full knowledge of the diffeomorphism $\chi$, only  the derivative $\partial_e \chi(0,\cdot)$.
\item
The fact that $\Lambda$ may be identified as a subset of the finite dimensional space $\R^{3\times N}$ will be  exploited to reduce the 
problem of differentiability of $\mathcal{E}$ to be an application of the 
Implicit Function Theorem applied to a reformulated interaction energy.
\item
The condition \eqref{eq:ConditionsForDiffeo} may be decomposed into three parts: $T(p+q)\left( v \circ \chi^{-1}\right) = T(p)v$ for all $v \in H^2(\Gamma)$, $\int_\Gamma v \circ \chi^{-1} = \int_\Gamma v$ for all $v \in H^2(\Gamma)$ and $v \in H^2(\Gamma) \iff v \circ \chi^{-1} \in H^2(\Gamma)$.
\item 
The condition on $\chi$ that $\int_\Gamma v\circ \chi(q,\cdot)^{-1} = \int_\Gamma v$ for all $v \in H^2(\Gamma)$ is equivalent to requiring that $\det\left( \nabla_\Gamma \chi(q,\cdot) + \nu(\cdot) \circ \chi(q,\cdot) \otimes \nu(\cdot) \right)=1$ on $\Gamma$.
As such, it is sufficient to have $\partial_e \det\left( \nabla_\Gamma \chi(q,\cdot) + \nu(\cdot) \circ \chi(q,\cdot) \otimes \nu(\cdot) \right)=0$ for any $e \in \prod_{i=1}^N \left( \R \times T_{X_{\mathcal{G}_i}} \Gamma\right)$ .
We will later see that, for $q = 0$, this is the same as requiring $\divg \partial_e \chi(0,\cdot)$ vanishes.
\end{itemize}
	\end{remark}

% a condition on the surface-divergence of $\partial_e \chi$.

%Since only the derivative (with respect to $p$) of these diffeomorphisms is required  we are able to provide  a computationally  accessible representation of this.

\subsection{The transformed functional and its derivative}

%	We will refer to this $p$ as the reference configuration.

Using the  $\chi$ satisfying  Assumption \ref{ass:ThereExistsDiffeo}, we have the following functional.
%It is our candidate for showing differentiability of $\mathal{E}$.
\begin{definition}
Let $\Jstar \colon \mathcal{B}\times U(p) \to \R$ be given by
\[
	\Jstar \colon (q,v) \mapsto J (v\circ \chi^{-1}(q,\cdot)),~J^*(0,v)=J(v)~\forall v\in U(p).
\]
We call  $J^*(\cdot,v)$ the transformed membrane energy.
Given $e \in \prod_{i=1}^N \left( \R \times T_{X_{\mathcal{G}_i}}\right)$, if, for any $v \in U(p)$, the derivative
\[
	\frac{\dee}{\dee t}\Jstar(te,v)|_{t=0}
\]
exists, we denote it $\partial_e \Jstar(0,v)$.
\end{definition}

We now define some terms which appear in \cite{ChuDjuEll19} which are useful to give an explicit representation of $\Jstar$.
\begin{definition}
	Given $q \in \mathcal{B}$, we define on $\Gamma$  the matrices and determinant
	\begin{align*}
		B=&B(q,\cdot) := \nabla_\Gamma \chi(q,\cdot) + \nu(\cdot) \circ \chi(q,\cdot) \otimes \nu(\cdot),
		\\
		G=&G(q,\cdot):=B(q,\cdot)^TB(q,\cdot),
		\\
		b=&b(q,\cdot):=\det(B(q,\cdot)).
	\end{align*}
\end{definition}

The following, convenient representation of $\Jstar$ is immediate from the results Lemmas \ref{lem:FirstOrderTrans} and \ref{lem:SecondOrderTrans} in the appendix.
\begin{lemma}
Given $v \in U(p)$, $q \in \mathcal{B}$,
% writing $B=B(q):= \nabla_\Gamma \chi(q,\cdot) + \nu\circ \chi(q,\cdot) \otimes \nu$, $b =b(q):= \det(B(q))$ and $G=G(q) := B(q)^T B(q)$,
it holds that
\begin{equation}\begin{split}
%J(u\circ X^{-1}(q))
\Jstar(q,v)=&\frac{\kappa}{2}\int_{\Gamma} \frac{1}{b}{\left(\divg (b G^{-1} \nabla_\Gamma v)\right)^2}%{\det(B(q))}
		+\left(\frac{\sigma}{2} - \frac{\kappa}{R^2}\right) \int_\Gamma b\nabla_\Gamma v \cdot G^{-1} \nabla_\Gamma v
		-\frac{\sigma}{R^2}\int_\Gamma b v^2.\end{split}
\end{equation}
\end{lemma}
Note that we wish to differentiate $\Jstar$ with respect to $q$ and that the $q$ dependence is located in the coefficients $B(q)$. % first we show it is indeed differentiable.
\begin{lemma}
	%Let $\Jstar(q,u)$ be as above with $\mathcal{B}$ a neighbourhood of $0$ in $\R^{ 3\times N}$,
	Suppose $\mathcal{B}\subset \prod_{i=1}^N(\R\times T_{X_{\mathcal{G}_i}}\Gamma)$ is sufficiently 
	small with $0 \in \mathcal{B}$ and $\chi\in C^k(\mathcal{B}\times\Gamma;\Gamma)$, then $\Jstar \in C^{k-2} (\mathcal{B}\times U(p);\R)$.
\end{lemma}
\begin{proof}
	It is clear from the expression for $\Jstar$ that it depends on $B$, the derivative of $B$ 
	and smoothly (in $H^2(\Gamma)$) on $v$.
	Since $B(0)= \I$, the identity matrix,  $B$ depends continuously on $q$ and $\det$ is a continuous map, 
	thus for a sufficiently small neighbourhood $\mathcal{B}\ni 0$, $\det(B(q)) > c >0$ it hold that $B$ is non-singular.
	Thus by smoothness of the integrand, we may apply the dominated convergence 
	theorem to obtain $\Jstar \in C^{k-2}(\mathcal{B}\times U(p);\R)$.
\end{proof}

\begin{theorem}\label{thm:applicationOfIFT}
	There exists  an open neighbourhood $\hat{\mathcal{B}}$ of $0$ in $\prod_{i=1}^N(\R\times T_{X_{\mathcal{G}_i}}\Gamma)$ such that $\mathcal{E}(p + \cdot) \in C^{k-2}(\hat{\mathcal{B}};\R)$.
	% and for all $q \in \hat{\mathcal{B}}$ with $|\alpha| \leq m-2$, it holds\marginpar{describe what we mean by $\nabla_p^\alpha$?}
%	\[
%		\nabla_p^\alpha \mathcal{J}(p+q)= \nabla_q^\alpha \hat{\mathcal{J}}(q).
%	\]
	In particular, for $k\geq 3$ and $u = \argmin_{v\in U(p)} J(v)$,
	\[
		\partial_e \mathcal{E}(p) = \partial_e \Jstar(0,u).
	\]
%	where $u = \argmin_{v \in U(p)} J(v)$.
\end{theorem}
%\marginpar{what do we mean by $\partial_u$? e.g. is it clear? Also Tobias's Thesis only has the result for the first derivative, is this all that can be done in this case? Might want to revisit this.}
\begin{proof}
	In the following we suppress the dependence on $p$ and write $u=u(p), U_0=U_0(p)$. Define    $\mathcal J\in C^{k-2}(\mathcal{B}\times U_0;\R)$ by $$\mathcal J(q,v) := \Jstar(q,u+v)~~\mbox{for}~~~ (q,v) \in \mathcal{B} \times U_0.$$
	For fixed $q$, $\mathcal J(q,\cdot)$ is a quadratic functional and by the definition of $u$ we have that the minimiser of the  functional $\mathcal J(0,v)$ over $U_0$ is given by $v=0$.
	Define  $F\in C^{k-2}(\mathcal{B}\times U_0;U_0^*)$ by $$ F(q,v):= D_v\mathcal J(q,v)$$
	 where, for fixed $q$, $D_v
	\mathcal J$ is  the first  variation of $\mathcal J(q,\cdot)$  over $U_0$. For each $(q,v)$, $F(q,v)$ is a linear functional.
	Since $J(0,v)$ attains minima at $v=0$, it follows that $F(0,0) = D_v\mathcal{J}(0,0) = 0 \in U_0^*$.
%	Furthermore, for fixed $q \in \mathcal{B}$, the first variation of $F$, 
%	\begin{align*}
%		D_{v} F(q,v)[\xi,\eta] = D_{vv}\mathcal{J}(q,v)[\xi,\eta]
%		=&
%		\kappa \int_\Gamma \frac{1}{b} \divg \left(bG^{-1}\nabla_\Gamma \xi \right) \divg \left(bG^{-1}\nabla_\Gamma \eta \right)
%		\\
%		&+
%		\left(\sigma-\frac{2\kappa}{R^2}\right) \int_\Gamma b \nabla_\Gamma \xi \cdot G^{-1} \nabla_\gamma \eta - \frac{2\sigma}{R^2} \int_\Gamma b \xi \eta,
%	\end{align*}
%	is a strictly coercive bilinear form.
%	As a consequence, it follows that the map $U_0\ni v \mapsto D_vF(q,v)\in U_0^*$ is invertible.
Furthermore, the first variation of $F$ at $(0,0)$,
\begin{align*}
	D_vF(0,0)\colon (\xi,\eta) \in U_0\times U_0 \mapsto D_v F(0,0)[\xi,\eta]
	= D_{vv} \mathcal{J}(0,0)[\xi,\eta]
	=&
	a(\xi,\eta),
%	\kappa \int_\Gamma \Delta_\Gamma \xi \Delta_\Gamma \eta
%	+
%	\left(\sigma-\frac{2\kappa}{R^2}\right) \int_\Gamma \nabla_\Gamma \xi \cdot \nabla_\Gamma \eta - \frac{2\sigma}{R^2}\int_\Gamma \xi \eta,
\end{align*}
is a strictly coercive bilinear form over $U_0\times U_0$.
As a consequence, it follows that the map $U_0\ni v \mapsto D_vF(0,v) \in U_0^*$ is invertible.
	%This follows from the fact that $D_{vv}F(q,v)$, the second variation of $F(q,\cdot)$, is a strictly coercive bilinear form.
	% (as a function $U_0(p) \to U_0(p)^*$) for $(q,v) \in \mathcal{B} \times U_0^*$ by virtue of $\Jstar_{vv}(q,v)$ being an invertible bilinear form.
	
	It therefore holds that we may apply the implicit function theorem, Theorem \ref{Thm:ImplictFnThm}, to $f=F$, with $(a,b) = (0,0)$, $\mathcal X= \prod_{i=1}^N (\R \times T_{X_{\mathcal{G}_i}}\Gamma)$, $\mathcal Y= U_0$, $\mathcal Z=\mathcal Y^*$ and $\Omega = \mathcal{B}\times \mathcal Y$.
	As such, there is neighbourhood of $0$, $\hat{\mathcal{B}}=V\subset \mathcal{B}$ and a function $\hat{v}\in C^{k-2}(\hat{\mathcal{B}};U_0(p))$ such that $\hat{v}(0)=0$ and $F(q,\hat{v}(q))=0$.
	That is to say $\Jstar_v(q,\hat{v}(q)+u) = 0$, so $\hat{v}(q)+u$ is a critical point of $\Jstar(q,\cdot)$.
	By coercivity of $\Jstar(q,\cdot)$ over $U(p)$, $\hat{u}(q):=\hat{v}(q)+u$ is the unique minimiser.
%	The continuity of $\Jstar$ and coercivity of $\Jstar(q,\cdot)$ imply that $\hat{u}(q):= \hat{v}(q)+u$ is the unique minimiser of $\Jstar(q,\cdot)$ over $U(p)$.
	Hence
	\[
		\mathcal{E}(p+q)
		= \min_{\eta\in U(p+q)}J(\eta)
		= \min_{\eta \in U(p)} \Jstar(q,\eta)
		= \Jstar(q,\hat{u}(q)).
	\]
	Since $\hat{u}\in C^{k-2}(\hat{\mathcal{B}};U(p))$, $\Jstar\in C^{k-2}(\mathcal{B}\times U(p);\R)$, it follows $\mathcal{E}(p+\cdot) \in C^{k-2}(\hat{\mathcal{B}};\R)$.
	Taking the derivative of $\mathcal{E}$ gives
	\[
		\partial_e \mathcal{E}(p)
		=
		\frac{\dee}{\dee t}\mathcal{E}(p+te)|_{t=0}
		=
		\frac{\dee}{\dee t}\Jstar(te,u)|_{t=0} + \frac{\dee}{\dee t}\Jstar(0,\hat{u}(te))|_{t=0}= \partial_e \Jstar(0,u),
	\]
	where $\frac{\dee}{\dee t}\Jstar(0,\hat{u}(te))|_{t=0} = D_v \Jstar(0,u)\left[ \frac{\dee}{\dee t}\hat{u}(te)|_{t=0}\right]$ vanishes since $D_v\Jstar(0,u)=0$.
\end{proof}%\marginpar{verify this proof, I am slightly uncertain}

\begin{remark}
Although  $\Jstar$  depends on the choice of $\chi$,
the derivative $\partial_e \mathcal{E}(p)$  is independent of the choice of $\chi$.
One may consider  a different diffeomorphism, say, $\tilde{\chi}$ with energy $\tilde{\Jstar}$, one would then have that
\[
	\min_{\eta\in U(p+q)} \Jstar(q,\eta) = \min_{\tilde{\eta} \in U(p+q)} \tilde{\Jstar}(q,\tilde{\eta})
\]
and arrive at $\partial_e \mathcal{E}(p) = \partial_e \tilde{\Jstar}(0,u) = \partial_e \Jstar(0,u)$.
\end{remark}

\subsection{An explicit formula for the derivative}
%\marginpar{Need to remove the divergence terms, might be worth having in a corollary in the above that $\divg V = 0$?}

%Although we work with area preserving maps, we include the terms where the divergence of the velocity appears as, for other applications, one may not have this property.
%Indeed for the numerical experiments which appear in Section \ref{sec:NumericalAnalysis}, we use the formula with these terms.

It is convenient to define the following.
\begin{definition}\label{def:DevineV}
	Define the tangential vector field $V\colon \prod_{i=1}^N\left(\R \times  T_{X_{\mathcal{G}_i}}\Gamma \right)\times \Gamma\to \R^3$ by
	\[
		V(e,x):= \partial_e \chi(0,x),
	\]
	which is tangential in the sense that $V(e,x) \in T_x \Gamma$ for all $(e,x) \in \prod_{i=1}^N\left(\R \times  T_{X_{\mathcal{G}_i}}\Gamma \right)\times  \Gamma$.
\end{definition}
\begin{proposition}%\marginpar{split into lemmas?}
\label{prop:RawDerivativeFormula}
	Given $e \in \prod_{i=1}^N\left(\R \times  T_{X_{\mathcal{G}_i}}\Gamma \right)$,
	%\prod_{i=1}^N (\R\times T_{X_{\mathcal{G}_i}}\Gamma)$,
	set $\mathcal{A}:= (\divg V)I - (\nabla_\Gamma V + \nabla_\Gamma V^T)$ then for $\eta\in H^2(\Gamma)$ %writing $u = u(p)$ % = \argmin_{v \in U(p)} J(v)$ 

	% - (\mathcal{H}V) \otimes \nu$.% - \nu \otimes (\mathcal{H}V)$.\marginpar{not sure if it is possible to remove the curvautre/normal terms here}
	\begin{align*}
		\partial_e \Jstar(0,\eta)
		=&
		\kappa \int_\Gamma  (\mathcal{A}: D_\Gamma^2 \eta - \Delta_\Gamma V \cdot \nabla_\Gamma \eta)\Delta_\Gamma \eta \\
		&
		-\frac{\kappa}{R^2}\int_\Gamma (V \cdot \nabla_\Gamma	 \eta + \frac{1}{2}\divg V  \Delta_\Gamma \eta)\Delta_\Gamma \eta
		%+\kappa\int_\Gamma ( \mathcal{H} ( \mathcal{H} - HI)V \cdot \nabla_\Gamma	 \eta - \frac{1}{2}\divg V  \Delta_\Gamma \eta)\Delta_\Gamma \eta
		\\
		&+
		\left(\frac{\sigma}{2} - \frac{\kappa}{R^2}\right)\int_\Gamma \nabla_\Gamma \eta \cdot\mathcal{A}\nabla_\Gamma \eta
		-
		\frac{\sigma}{R^2}\int_\Gamma \divg V \eta^2.
	\end{align*}
\end{proposition}
\begin{proof}
	We will make use of the fact that $B(0) = \I$ and $\det(B(0)) =1$.
	To simplify notation when taking derivative $\partial_e$, we assume that we are evaluating at $q=0$, if there is no argument given.
	The product rule gives
	\begin{equation}\label{eq:FirstForcingProof1}\begin{split}
		\partial_e \Jstar(0,\eta) =& \frac{\kappa}{2}\int_\Gamma 2 \divg \frac{\dee}{\dee t}(\det(B(te)) G(te)^{-1} \nabla_\Gamma \eta)|_{t=0}\Delta_\Gamma \eta -  \left(\Delta_\Gamma \eta \right)^2 \frac{\dee}{\dee t}\det(B(te))|_{t=0}
		\\
		&+ \left( \frac{\sigma}{2} -\frac{\kappa}{R^2}\right)\int_\Gamma \nabla_\Gamma \eta \cdot \frac{\dee}{\dee t}(\det(B(te)) G(te)^{-1})|_{t=0} \nabla_\Gamma \eta
		\\
		&-\frac{\sigma}{R^2}\int_\Gamma \frac{\dee}{\dee t}\det(B(te))|_{t=0} \eta^2.
	\end{split}\end{equation}
	Where we calculate
	\begin{align*}
		\partial_e B
%		=&
%		\partial_e \left(\nabla_\Gamma X(q) + (\nu\circ X(q))\otimes \nu \right)|_{q=0}
%		\\
%		=&
%		\nabla_\Gamma \partial_e X(q)|_{0=0} + \partial_e(\nu \circ X(q))|_{q=0}\otimes \nu
%		\\
		=&
		\nabla_\Gamma V + (\mathcal{H} V) \otimes \nu,
		\\
%	\end{align*}
%	This then gives
%	\begin{align*}
		\partial_e \det(B) =& %\det(B(q))\text{Tr}(B(q)^{-1} \partial_e B(q) )|_{q=0}
		%\\
		%=&
		%\text{Tr}(\partial_e B(0) ) = 
		\divg V,
		\\
%	\end{align*}
%	where we make use of $\text{Tr}( (\mathcal{H}V)\otimes \nu) = \nu\cdot \mathcal{H}V =0$.
%	\begin{align*}
	\partial_e B^{-1}=&% - B(q)^{-1} \partial_e B(q) B(q)^{-1}|_{q=0}
%	\\
%	=&
%	-\partial_e B(0) =
	 -\nabla_\Gamma V - (\mathcal{H}V) \otimes \nu.
	\end{align*}
	Since $G := B^T B$ one has,
	\begin{align*}
		\frac{\dee}{\dee t} (\det(B(te))G(te)^{-1})|_{t=0}
%		=&
%		\partial_e \det(B(q))|_{q=0} \,I + \partial_e B(q)^{-1}|_{q=0}
%		+
%		\partial_e B(q)^{-T}|_{q=0}
%		\\
		=& (\divg V)\, \I -\nabla_\Gamma V - (\mathcal{H}V) \otimes \nu -\nabla_\Gamma V^T - \nu\otimes(\mathcal{H}V).
	\end{align*}
	We are also required to calculate the surface divergence of the above quantity% when evalutated on the surface gradient of a (sufficiently smooth) function,
	\begin{align*}
		\divg \partial_e (\det(B)G^{-1}&)
		\\
		=&
		\divg \left((\divg V)\, \I -\nabla_\Gamma V - (\mathcal{H}V) \otimes \nu -\nabla_\Gamma V^T - \nu\otimes(\mathcal{H} V)\right)
%		\\
%		=&
%		\divg ( \I \divg V ) - \Delta_\Gamma V  - \nabla_\Gamma\cdot (\nabla_\Gamma V^T) -  \divg (\mathcal{H}V \otimes \nu + \nu \otimes \mathcal{H}V)
		\\
		=&
		\sum_{k=1}^{n+1} (\nabla_\Gamma \underline{D}_k - \underline{D}_k\nabla_\Gamma )V_k - \Delta_\Gamma V - \divg( (\mathcal{H}V) \otimes \nu + \nu \otimes (\mathcal{H}V) )
%		\\
%		=&
%		-\Delta_\Gamma V
%		+ \sum_{k=1}^{n+1} (\mathcal{H}\nabla_\Gamma V_k)_k \nu - (\mathcal{H}\nabla_\Gamma V_k) \nu_k - \underline{D}_k(\nu (\mathcal{H}V)_k ) -\underline{D}_k (\nu_k \mathcal{H}V)
		\\
		=&
		-\Delta_\Gamma V - H \mathcal{H}V
		+ \sum_{k=1}^{n+1} (\mathcal{H}\nabla_\Gamma V_k)_k \nu - (\mathcal{H}\nabla_\Gamma V_k) \nu_k - \underline{D}_k(\nu (\mathcal{H}V)_k ) .
	\end{align*}
	It is possible to see
	\[
		\sum_{k=1}^{n+1} (\mathcal{H}\nabla_\Gamma V_k)_k
%		=
%		\sum_{k=1}^{n+1} \sum_{l=1}^{n+1} \mathcal{H}_{lk} \underline{D}_l V_k
		=
		\mathcal{H}: \nabla_\Gamma V,
	\]
	by using that $V \cdot \nu =0$,
	\begin{align*}
		\sum_{k=1}^{n+1} (\mathcal{H}\nabla_\Gamma V_k)_j \nu_k
%		=
%		\sum_{k=1}^{n+1} \sum_{l=1}^{n+1} \mathcal{H}_{l j}\underline{D}_l V_k \nu_k
%		=
%		\sum_{k=1}^{n+1} \sum_{k=1}^{n+1} \mathcal{H}_{l j} \underline{D}_l(V_k \nu_k) - \mathcal{H}_{l j} \mathcal{H}_{lk} V_k
		=
		-(\mathcal{H}^2 V)_j.
	\end{align*}
	Furthermore,
	\begin{align*}
		\sum_{k=1}^{n+1} \underline{D}_k (\nu (\mathcal{H}V)_k)
		=&
%		\sum_{k=1}^{n+1} \sum_{l=1}^{n+1}\underline{D}_k(\nu  \mathcal{H}_{kl}V_l )
%		=
%		\sum_{k=1}^{n+1} \sum_{l=1}^{n+1} \underline{D}_k \nu \mathcal{H}_{kl}V_l + \nu \underline{D}_k (\mathcal{H}_{kl}V_l)
%		\\
%		=&
%		\sum_{k=1}^{n+1} \sum_{l=1}^{n+1} \underline{D}_k \nu \mathcal{H}_{kl}V_l + \nu \underline{D}_k \mathcal{H}_{kl} V_l +\nu \mathcal{H}_{kl} \underline{D}_k V_l
%		\\
%		=&
		\mathcal{H}^2 V + (\mathcal{H}: \nabla_\Gamma V + (\nabla_\Gamma \cdot \mathcal{H})\cdot V)\nu.
	\end{align*}
%	and
%	\[
%		\sum_{k=1}^{n+1} \underline{D}_k (\nu_k \mathcal{H}V)
%		=
%		\sum_{k=1}^{n+1} \underline{D}_k \nu_k \mathcal{H}V + \nu_k \underline{D}_k (\mathcal{H}V)
%		=
%		H \mathcal{H}V.
%	\]
	Together this gives,
	\begin{align*}
		\divg(\partial_e (\det(B)G^{-1}))
		=
		-\Delta_\Gamma V - \nu (\nabla_\Gamma \cdot \mathcal{H})\cdot V - H \mathcal{H}V,
	\end{align*}
	where the middle term will vanish when multiplied against a tangential vector field.
	We are left with
	\begin{align*}
		\partial_e (\det(B(q))G^{-1}) : D_\Gamma^2 \eta
		=
		\mathcal{A}:D_\Gamma^2 \eta &- (\mathcal{H}V)\otimes \nu :D_\Gamma^2 \eta \\
		&- \nu \otimes (\mathcal{H}V) : D_\Gamma^2 \eta,
	\end{align*}
	where one may recall that for $b,c$ vectors and matrix $A$, $A:(b\otimes c) = b^T A c$.
	Thus
	\[
		\partial_e (\det(B)G^{-1}) : D_\Gamma^2 \eta
		= 
		\mathcal{A}:D_\Gamma^2 \eta + \mathcal{H}^2\nabla_\Gamma \eta \cdot V,
	\]
	which completes the result when evaluating $H$ and $\mathcal{H}$ for a sphere.
%	Therefore, one has
%	\begin{align*}
%		\partial_e \Jstar(u,q)|_{q=0}
%		=&
%		\kappa \int_\Gamma  (\mathcal{A}: D_\Gamma^2 u - \Delta_\Gamma V \cdot \nabla_\Gamma u  + \mathcal{H} ( \mathcal{H} - H\id{\R^3})V \cdot \nabla_\Gamma	 u - \frac{1}{2}\divg V  \Delta_\Gamma u)\Delta_\Gamma u
%		\\
%		&+
%		\frac{1}{2}\left(\sigma - \frac{2\kappa}{R^2}\right)\int_\Gamma \nabla_\Gamma u \mathcal{A}\nabla_\Gamma u
%		-
%		\frac{\sigma}{R^2}\int_\Gamma \divg V u^2
%	\end{align*}
%	as required.
\end{proof}

By Theorem \ref{thm:applicationOfIFT}, when evaluating this at the solution of Problem \ref{prob:SingleParticleEnergyMinimisation}, we will obtain the derivative we seek.
We notice that it might be convenient to integrate by parts to remove the surface Hessian.  %this gives a term with $\nabla_\Gamma \cdot \mathcal{A}$ which we have calculated above to be $-\Delta_\Gamma V + (\mathcal{H}:\nabla_\Gamma V )\nu + \mathcal{H}^2 V$.
This will give an alternate formula which is better suited for the numerical methods considered in \cite{EllFriHob19,EllHer20-A}.
%This may be done in the first step of the above proof \eqref{eq:FirstForcingProof1} to give the follow corollary.
\begin{corollary}\label{cor:ReformulationSOS}
	Under the assumptions of Proposition \ref{prop:RawDerivativeFormula} it may be seen that, for $\eta \in W^{3,p}(\Gamma),\, p<2$,
	%\marginpar{we had an extra term (the curvature term!) Were simulations good with it?  is the mistake to removing it? {\color{red} check stuff :(}}
	\begin{equation}\label{eq:DerivativeFullForm}\begin{split}
		\partial_e \Jstar(q,\eta)|_{q=0}
		=&
		-\kappa \int_\Gamma %-H \mathcal{H}V \cdot \nabla_\Gamma v
		\frac{1}{2}\left(\divg V \right)\left(\Delta_\Gamma \eta\right) ^2 +  \nabla_\Gamma \Delta_\Gamma \eta\cdot \mathcal{A} \nabla_\Gamma \eta
		\\
		&+
		\frac{1}{2}\left(\sigma - \frac{2\kappa}{R^2}\right)\int_\Gamma \nabla_\Gamma \eta \cdot \mathcal{A}\nabla_\Gamma \eta
		-
		\frac{\sigma}{R^2}\int_\Gamma \left(\divg V\right) \eta^2.
	\end{split}\end{equation}
%	In particular, as we are concerned with the sphere of radius $R$,
%	\begin{align*}
%		\partial_e \Jstar(q,u)|_{q=0}
%		=&
%		\kappa \int_\Gamma \left(%-\frac{2}{R^2} V \cdot \nabla_\Gamma u 
%		- \frac{1}{2}\nabla_\Gamma \cdot V \Delta_\Gamma u \right)\Delta_\Gamma u - \nabla_\Gamma \Delta_\Gamma u \cdot \mathcal{A} \nabla_\Gamma u
%		\\
%		&+
%		\frac{1}{2}\left(\sigma - \frac{2\kappa}{R^2}\right)\int_\Gamma \nabla_\Gamma u \cdot \mathcal{A}\nabla_\Gamma u
%		-
%		\frac{\sigma}{R^2}\int_\Gamma \nabla_\Gamma\cdot V u^2.
%	\end{align*}
	\begin{comment}
	Furthermore, since we have $\divg V =0$,
	\begin{equation}\label{eq:DerivativeReduced}\begin{split}
		\partial_e \Jstar(q,\eta)|_{q=0}
		=&
		\kappa \int_\Gamma \nabla_\Gamma \Delta_\Gamma \eta \cdot \left(\nabla_\Gamma V + \nabla_\Gamma V^T\right)\nabla_\Gamma \eta
		-\left(\sigma - \frac{2\kappa}{R^2}\right)\int_\Gamma \nabla_\Gamma \eta \cdot\nabla_\Gamma V\,\nabla_\Gamma \eta.
	\end{split}\end{equation}
	\end{comment}
\end{corollary}
\begin{proof}
This follows from integration by parts in \eqref{eq:FirstForcingProof1} and following through with the proof above.
The integration by parts is admissible by the regularity of $\eta$.% shown in Theorem \ref{thm:ExistenceRegularity}.
\end{proof}

%as $u$ will not have the required global regularity, as such this formula would need to be accompanied by a line integral.
%\[
%	\int_\Gamma \mathcal{A}: D_\Gamma^2 u \Delta_\Gamma u
%	=
%	\int_\Gamma 
%\]
%We note that the assumption $\divg V = 0$ is satisfied for our choice of $X$.

%\subsection{'Conclusion'}\marginpar{Some way to point out the 'big result'}
%With the work of the preceding subsections,
By the additional regularity shown in Theorem \ref{thm:ExistenceRegularity}, we see that we may pick $\eta = \argmin_{v \in U(p)}J(v)$ in the above.
This gives the main result of the work which follows from the previous results.
\begin{theorem}
	Let $p \in \Lambda^\circ$, $u = \argmin_{v\in U(p)} J(v)$ and $\mathcal{A}:= (\divg V)I - \nabla_\Gamma V - \nabla_\Gamma V^T$, then %for $V$ is as in Corollary \ref{cor:SimpleV},
	\begin{equation}\label{eq:explicitDerivativeFormula}\begin{split}
		\partial_e \mathcal{E}(p)
		=&
		-\kappa \int_\Gamma %-H \mathcal{H}V \cdot \nabla_\Gamma v
		\frac{1}{2}\left(\divg V \right)\left(\Delta_\Gamma u\right) ^2 +  \nabla_\Gamma \Delta_\Gamma u\cdot \mathcal{A} \nabla_\Gamma u
		\\
		&+
		\frac{1}{2}\left(\sigma - \frac{2\kappa}{R^2}\right)\int_\Gamma \nabla_\Gamma u \cdot \mathcal{A}\nabla_\Gamma u
		-
		\frac{\sigma}{R^2}\int_\Gamma \left(\divg V\right) u^2.\end{split}
	\end{equation}
\end{theorem}
\begin{proof}
This is an application of Theorem \ref{thm:applicationOfIFT} and Corollary \ref{cor:ReformulationSOS}.
\end{proof}

\begin{corollary}\label{cor:ZeroDerivative}
Let $N = 1$, then $\partial_e \mathcal{E}(p) = 0$ for all $p\in \Lambda^\circ$ and directions $e \in \R \times T_{X_\mathcal{G}} \Gamma$.
\end{corollary}
\begin{proof}
This result follows from the symmetry of the sphere and the  invariance of $J$ under rotations and translations.
\end{proof}

\subsection{Transformations satisfying Assumption \ref{ass:ThereExistsDiffeo}}\label{subsec:ConstructDiffeo}
Here, we verify Assumption \ref{ass:ThereExistsDiffeo} by constructing $\chi$.
%The first construction we present is motivated by the $V$ we use for our more complex experiments, whereas 
\subsubsection{Rotation of a single particle}\label{subsubsec:simpleChi}
This example pertains to a simple rotation.
The example we consider is rotating a single  particle whose centre $X_G$ is taken to be   the North pole $N:=(0,0,R)^T$ without loss of generality. The points of the particle are contained in the set $B_r(N):=\{ x: x_3>R-r\}$ around the North pole and all other points are contained in the set $B_{r+\epsilon}(N)^C:=\{ x : x_3 <R-r-\epsilon\}$.

Since this is a 1-parameter family of transformations, we write, with an abuse of notation $\chi(\alpha,\cdot)= \chi(q,\cdot)$ for the diffeomorphism.

We may then explicitly write
\begin{align*}
	\chi(\alpha,x) =&
	\eta(x) \left((0,0,x_3)^T + \cos(\alpha) \left( \frac{N}{R} \times x\right) \times \frac{N}{R} + \sin(\alpha) \left( \frac{N}{R} \times x \right) \right)
	+ \left(1-\eta(x)\right))x,
\end{align*}
where $\eta\colon\Gamma\to \R$ is a $C^{k}$-smooth cut off function such that $\eta = 1$ on $B_r(N)$ and $\eta=0$ on $B_{r+\epsilon}(N)^C$ and depends only on $x_3$.

It is clear that this $\chi$ is smooth with $\chi(\alpha,\cdot)$ having inverse $\chi(-\alpha,\cdot)$ and
that  it moves the points of the particle based at the north pole as required, while others remain stationary.
Furthermore,  for each fixed $x_3$ it ,essentially, is a 2-dimensional rotation about $(0,0,x_3)$ so the volume element induced by $\chi$ is constantly equal to 1.

It is convenient to calculate, for $e = (1,0)$, $\partial_e \chi(0,x)$,
\[
	\partial_e \chi(0,x)=\partial_s \left(\chi(s,x)\right)|_{s=0}
	=
	\eta(x)  \left( \frac{N}{R} \times x\right).
\]
One may also verify that $\divg \partial_e \chi(0,\cdot) = 0$.
This follows by calculating
\begin{align*}
	\divg \partial_e \chi(0,x)
	=
	\frac{1}{R} \left( \nabla_\Gamma \eta(x) \cdot \left( N \times x \right) + \eta(x) \divg (N\times x)\right),
\end{align*}
by the fact that $\eta$ depends only on $x_3$, one sees that the first term is some scalar function multiplied by $P_\Gamma(x) N \cdot \left( N\times x\right)$, which vanishes.
For the second term, one calculates, by extending to a small neighbourhood of the surface (as in the definition of surface derivatives),
\begin{align*}
	\divg \left( N \times x\right)
	=
	\sum_{i=1}^3 \D{i}\left( N \times x\right)_i
	=
	\sum_{i,j=1}^3 \left( \delta_{ij} -\frac{x_i x_j}{R^2}\right)\partial_j\left( N \times x\right)_i.
%	=
%	-\sum_{i,j=1}^3 \frac{x_i x_j}{R^2}\partial_j \left( N \times x\right)_i
%	=0
\end{align*}
We see that this vanishes, since $\delta_{ij}\partial_j(N \times x)_i = 0$ for any $i,j=1,2,3$, and
\[
	\sum_{i=1}^3 \frac{x_i x_j}{R^2} \partial_j (N \times x)_i = \sum_{i=1}^3 \frac{x_j}{R^2} \partial_j \left(x_i (N \times x)_i \right) =0
\]
for any $j=1,2,3$.

%One may also consider a similar construction to consider translations for the particle at the north pole with translation vector $\tau = (1,0,0)^T$, say.
%This is done by replacing $\eta$ with $\xi$, a function of $x_1$, which has support .
%This also requires no other particles are within the support of $\xi$.

%This construction for $\chi$ is very simple, but applies to very few cases, if the particle configurations allow it, one may use linear combinations of the above construction to make more complex transformations for other
\subsubsection{A general $\chi$}
Since the set $\bigcup_{i=1}^N\cee_i(p)$ is a finite union of points, we know there is a strictly positive distance separating  each pair of  points.
It follows that we may  assume that the family of sets $\bigcup_{i=1}^N\cee_i(p+tq)$ for $(t,q) \in [0,1]\times \mathcal{B}$ also satisfy this condition,
and set $\epsilon>0$ to be the smallest separation between the points of $\bigcup_{i=1}^N\cee_i(p+tq)$ - that is
\[
	\epsilon = \inf_{(t,q)\in [0,1]\times\mathcal{B}} \inf_{x \in \bigcup_{i=1}^N\cee_i(p+tq)} \inf_{y \in \bigcup_{i=1}^N\cee_i(p+tq), y\neq x} |x-y|.
\]

\begin{definition}[Equation (2.6) \cite{Reu20}]
	We define the vector surface curl of a $C^1$ function $\psi\colon \Gamma \to \R$ by
	\[
	\curl_\Gamma \psi:= \nu\times \nabla_\Gamma\psi.
	\]
\end{definition}

\begin{definition}\label{CurlyV}
Given $\delta\in (0,\epsilon)$, define $\mathcal{V}\colon [0,1]\times \mathcal{B}\times \Gamma \to \R^3$ by
\[
	\mathcal{V}:=\curl_\Gamma \psi
\]
where for each $(t,q) \in [0,1]\times \mathcal{B}$,  $x \in \bigcup_{i=1}^N\cee_i(p+tq)$, the function $\psi\colon [0,1]\times \mathcal{B}\times \Gamma\to \R$ is given by
\[
	\psi(t,q,y) = \eta(|x-y|) y \cdot ( \partial_s \left( \phi_i(p+sq,\cdot) \circ \phi_i(p+tq,\cdot)^{-1}(y)  \right)|_{s=t}\times \nu(x) )
\]
for $y\in \Gamma\cap B_{\epsilon/2}(x)$, otherwise $\psi = 0$, where $\eta\colon \R\to \R$ is a $C^{k+1}$-smooth cut off function such that
\[
	\begin{cases}
		\eta(s) = 1  &|s| \leq \delta/4,
	\\
		\eta(s) = 0  &|s| \geq \delta/2.
	\end{cases}
\]
\end{definition}

\begin{example}
We now give a calculation of $\partial_s \left( \phi_i(p+sq,\cdot) \circ \phi_i(p+tq,\cdot)^{-1}(y)  \right)|_{s=t}$.
For simplicity, we set $p=0$ and $t=0$ and neglect any $i$ subscripts.

Let $q = (\alpha,\tau)\in \R\times T_{X_\mathcal{G}}$.
We then have
\[
	\phi(sq,x) = R_T(s\tau) R_n(s\alpha)x,
\]
therefore
\[
	\partial_s \left( \phi(sq,x)\right)|_{s=0} = \left( \nu(X_\mathcal{G})\times \tau \right)\times x + \alpha \left( \nu(X_\mathcal{G}) \times x\right).
\] 
It is clear that the first term corresponds to the translation and the second term the rotation.
\end{example}
\begin{comment}
For each $(t,q) \in [0,1]\times \mathcal{B}$, in the ball where $\eta=1$ around each of the points $x\in\bigcup_{i=1}^N\cee_i(p+tq)$, we have that both
\[
	\psi(t,q,y) = y \cdot ( \partial_s \left( \phi_i(p+sq,\cdot) \circ \phi_i(p+tq,\cdot)^{-1}(y)  \right)|_{s=t} \times \nu(x))
\]
and
\[
	\mathcal{V}(t,q,y)
	= \nu(y) \times\left( ( \partial_s \left( \phi_i(p+sq,\cdot) \circ \phi_i(p+tq,\cdot)^{-1}(y)  \right)|_{s=t} \times \nu(x)\right),
\]
so clearly when $y=x$,
\[
	\mathcal{V}(t,q,x) = ( \partial_s \left( \phi_i(p+sq,\cdot) \circ \phi_i(p+tq,\cdot)^{-1}(x)  \right)|_{s=t},
\]
where these values may be explicitly determined from the definitions of $\phi_i$ in Definition \ref{def:MultiplePointsMove}.
For $t=0$, $\tau \in T_{X_{\mathcal{G}_i}}\Gamma$, $\tau\neq 0$, $q_i = (0,\tau)$, one calculates, for $y$ with $\eta(y)=1$,
\[
	\mathcal{V}(0,q,y) = |\tau| \left(\left(\nu(X_{\mathcal{G}_i})\times \frac{\tau}{|\tau|} \right)\times y  \right),
\]
in particular, if $y= X_{\mathcal{G}_i}$, satisfies $\eta(X_{\mathcal{G}_i})=1$,
\[
	\mathcal{V}(0,q,X_{\mathcal{G}_i}) = \tau.
\]
\end{comment}

\begin{lemma}\label{lem:PropsOfV}
	The function $\mathcal{V}$ given in Definition \ref{CurlyV} satisfies:
	\begin{itemize}
	\item	$\mathcal{V}\in C^k$,
	\item	$\Div_\Gamma \mathcal{V} =0$,
	\item	$\mathcal{V}(t,0,x) = 0$ for all $(t,x) \in [0,1]\times \Gamma$,
	\item	for each $i=1,...,N$, $\mathcal{V}(t,q,\cdot) = \partial_s\left( \phi_i(p+ tq,\cdot) \circ \phi_i(p+sq,\cdot)^{-1} \right)|_{s=t}$ on $\cee_i(p + tq)$, for each $(t,q) \in [0,1]\times \mathcal{B}$,
	\item	$\partial_e \mathcal{V}(t,0,x) = \mathcal{V}(0,e,x)$ for all $t \in [0,1]$, $e \in \prod_{i=1}^N\left(T_{X_{\mathcal{G}_i}} \times \R \right)$ and $x \in \Gamma$.
	\end{itemize}
\end{lemma}
%\begin{comment}
\begin{proof}
	Smoothness and that $\mathcal{V}(\cdot,0,\cdot)$ vanishes is clear by construction, divergence free follows from $\mathcal{V}$ being the curl of another function \cite[Lemma 2.1]{Reu20}.
	For the point conditions we evaluate at $y\in \Gamma$ such that $|x-y|< \frac{\delta}{4}$ for some $x \in \cee_i(p+tq)$,
	\begin{align*}
		\curl_\Gamma \psi(t,q,y)
		=&
		\curl_\Gamma \left( y \cdot ( \partial_s \left( \phi_i(p+sq,\cdot) \circ \phi_i(p+tq,\cdot)^{-1}(y)  \right)|_{s=t})\times \nu(x) \right)
		\\
		=&
		\nu(y) \times \left( \nabla_\Gamma y \cdot \left(\partial_s\left( \phi_i(p+sq,\cdot) \circ \phi_i(p+tq,\cdot)^{-1}(y) \right) |_{s=t} \times \nu(x)\right)\right)
	\end{align*}
	for each $(t,q) \in [0,1]\times\mathcal{B}$, $i=1,...,N$.
	Which upon evaluation of at any $x \in \cee_i(p+tq)$, $(t,q)\in [0,1]\times \mathcal{B}$, $i=1,...,N$, leaves us with
	\[
		\curl_\Gamma \psi(t,q,x)
		=
		\partial_s\left( \phi_i(p+ sq,\cdot) \circ \phi_i(p+tq,\cdot)^{-1} \right)|_{s=t}(x).
	\]
	
	The final condition takes a little bit of work.
	We show the condition near the 'special points' of $\bigcup_{i=1}^N\cee_i(p)$.
		Given $i =1,...,N$, for $x \in \cee_i(p)$ and $y$ near $x$, we see that
		\begin{align*}
			\partial_e \mathcal{V}(t,0,y)
			=& \partial_s \mathcal{V}(t,se,y)|_{s=0}
			\\
			=& \partial_s \left( \mathcal{V}(t,se, \phi_i(p+se,\cdot) \circ \phi_i(p,\cdot)^{-1}(y))\right)|_{s=0}
			\\
			&+\partial_s\left( \mathcal{V}(t,se,x) -\mathcal{V}(t,se, \phi_i(p+se,\cdot) \circ \phi_i(p,\cdot)^{-1}(y))\right)|_{s=0}
			\\
			=&
			 \partial_s \left( \mathcal{V}(t,se, \phi_i(p+se,\cdot) \circ \phi_i(p,\cdot)^{-1}(y)) \right)|_{s=0}
			 \\
			 &+
			 \partial_s \left( \mathcal{V}(t,se,y) -\mathcal{V}(t,se, \phi_i(p+se,\cdot) \circ \phi_i(p,\cdot)^{-1}(y))\right)|_{s=0}.
		\end{align*}
		This first term we may see is equal to $ \mathcal{V}(0,e,x)$, for the remaining terms,
		%%to find this "equal to" one just applies the definition of what $\mathcal{V}$ is%%
		\begin{align*}
			\partial_s& \left( \mathcal{V}(t,se, \phi_i(p+se,\cdot) \circ \phi_i(p,\cdot)^{-1}(y))\mathcal{V}(t,se,y)\right)|_{s=0}
			\\&=
			\partial_s \left( \nabla_\Gamma \mathcal{V}(t,se,y) \cdot \left( \phi_i(p+se,\cdot) \circ \phi_i(p,\cdot)^{-1}(y)\right) -y \right)|_{s=0},
		\end{align*}
		which we see vanishes due to the fact that $\nabla_\Gamma \mathcal{V}(\cdot,se,\cdot) \to 0$ as $s\to 0$ on $[0,1]\times \Gamma$ and also $\phi_i(p+se,\cdot) \circ \phi_i(p,\cdot)^{-1}(y) -y \to 0$ as $s\to 0$.
\end{proof}
%\end{comment}

% Let $\mathcal{V}\colon [0,1]\times \mathcal{B}\times \Gamma \to \R^3$,  for each $q \in \mathcal{B}$ and each point on $\Gamma$ consider the following  ODE.

We will construct $\chi$ in the following way.

\begin{definition}\hspace{0.1mm}
	\begin{enumerate}
	\item
	Let $\eta\colon[0,1] \times \mathcal{B}\times \Gamma \to \Gamma$ be the solution of the family of ODEs
	\[
		\partial_t \eta(t,q,x) = \mathcal{V}(t,q,\eta(t,q,x)), ~ \eta(0,q,x) = x
	\]
	for all $(q,x) \in \mathcal{B}\times \Gamma$.
%\end{definition}
%\begin{definition}
\item
Let  $\chi\colon \mathcal{B}\times \Gamma \to \Gamma$ by $\chi(q,x) = \eta(1,q,x)$ for all $(q,x) \in \mathcal{B} \times \Gamma$.
\end{enumerate}
\end{definition}
%First we construct a right hand side for the ODE.

%We now define this right hand side.
%It is clear that, by construction, $\mathcal{V}$ has the following properties:

It is clear by standard ODE theory \cite{Har02} that $\eta$ exists and is smooth, furthermore, it is clear that $\eta(1,q,\cdot)$ is a diffeomorphism.

\begin{proposition}
	The map $\chi\colon \mathcal{B}\times\Gamma\to\Gamma$ $(q,x)\mapsto \eta(1,q,x)$ satisfies Assumption \ref{ass:ThereExistsDiffeo}.
\end{proposition}
\begin{proof}
This follows from the properties of $\mathcal{V}$ in Lemma \ref{lem:PropsOfV}.
The smoothness of $\chi$ follows from the smoothness of $\mathcal{V}$ and standard ODE theory \cite{Har02}, as does the existence and smoothness of an inverse.
The condition that $\mathcal{V}(\cdot,0,\cdot)=0$ gives that $\chi(0,\cdot)$ is the identity.

The condition $v \circ \chi(q,\cdot)^{-1} \in U(p+q) \iff v \in U(p)$ has three parts:
\begin{itemize}
\item	$v \circ \chi(q,\cdot) \in H^2(\Gamma) \iff v \in H^2(\Gamma)$,
\item	$\int_\Gamma v = \int_\Gamma v \circ \chi(q,\cdot)$ for all $v \in H^2(\Gamma)$,
\item	$T(p+q)\left(v \circ \chi^{-1}\right) = T(p) v$ for all $v \in H^2(\Gamma)$.
\end{itemize}
The first condition follows from two applications of Lemma \ref{lem:SecondOrderTrans} with $X = \chi(q,\cdot)$ and $X = \chi(q,\cdot)^{-1}$ and the smoothness of these maps.
The second condition follows from the fact that $\Div_\Gamma \mathcal{V}=0$.
% and the Leibniz rule \cite[Theorem 5.1]{DziEll13} applied to $\int_\Gamma v \circ \chi(q,\cdot).
The final condition follows from the point conditions on $\mathcal{V}$.
By considering the ODE that $\eta$ solves, we see that $\chi$ satisfies for each $i=1,...,N$,
\[
	\chi(q,\cdot) = \phi_i(p+q,\cdot) \circ \phi_i(p,\cdot)^{-1} \mbox{ on }\cee_i(p),
\]
which gives, recalling the definition of $T$ in Definition \ref{def:DefineT},
\begin{align*}
	T(p+q)_i v &= v \circ \phi_i(p+q,\cdot)|_{\cee_i}
	\\
	&=
	v \circ \phi_i(p+q,\cdot) \circ \phi_i(p,\cdot)^{-1} \circ \phi_i(p,\cdot)|_{\cee_i}
	\\
	&=
	T(p)_i(v \circ \chi(q,\cdot)).
\end{align*}
\end{proof}

We now wish to calculate $\partial_e \chi(0,\cdot)$ on $\Gamma$.
\begin{proposition}
	For each $e \in \prod_{i=1}^N\left(T_{X_{\mathcal{G}_i}} \times \R \right)$, the following formula holds
	$$\partial_e\chi(0,\cdot) = \mathcal{V}(0,e,\cdot)~~\mbox{ on }~~\Gamma.$$
\end{proposition}
\begin{proof}
	It is clear that $\partial_e \chi(0,\cdot) = \partial_e \eta(1,0,\cdot)$.
	From the ODE $\eta$ solves, one may see that $\eta_e(t,x) := \partial_e \eta(t,0,x)$ for $(t,x) \in [0,1]\times\Gamma$ satisfies
	\[
		\partial_t \eta_e (t,x) = \partial_e \mathcal{V}(t,0,\eta(t,0,x)) + \nabla_\Gamma \mathcal{V}(t,0,\eta(t,0,x))  \eta_e(t,x),
	\]
	for all $(t,x) \in [0,1]\times \Gamma$.
	Recall that $\mathcal{V}(t,0,x) = 0$ for all $(t,x) \in [0,1]\times \Gamma$, so the second term in the above ODE vanishes and one has that $\eta(t,0,x) = x$ for all $(t,x)\in [0,1]\times \Gamma$.
	By applying the final condition of Lemma \ref{lem:PropsOfV}, one has that
	\[
		\partial_t \eta_e (t,x) = \mathcal{V}(0,e,x),
	\]
	hence $\partial_e\chi(0,\cdot) = \eta_e(1,\cdot) = \mathcal{V}(0,e,\cdot)$ on $\Gamma$.
\end{proof}

\section{Numerical experiments}\label{sec:NumericalAnalysis}
%\marginpar{should remove the comments regarding Lipschitz $V$ being the maximal requirement}
We are now equipped to present some simulations, but first we discuss the approximation errors which arise in numerical simulations.% some of the numerical challenges for the approximation.
%\subsection{Approximation sensitivity}
\begin{comment}
\marginpar{this needs to be corrected to be relevant}
\begin{proposition}\label{prop:ApproxSensitivity}
	Let $\tilde{u}\in W^{1,q}(\Gamma)$ for any $q \in \mathcal{B}$ with $-\Delta_\Gamma \tilde{u}\in W^{1,2-\delta}(\Gamma)$ for any $\delta>0$ be an approximation of $u= u(p)$ and $V$ as discussed in the previous section.
	Let $\partial_e\tilde{J}(0,\tilde{u})$ be the evaluation this approximation $\tilde{u}$.
	Then for any $t \in(1,2)$ and $\epsilon\in (0,1)$, there is $C>0$ (independent of $\epsilon$ and $t$) such that
	\begin{align*}
		\partial_e \tilde{J}(0,\tilde{u}) - \partial_e \Jstar(0,u)
		\leq&
		C\big(\|V-V\|_{1,2}\|u\|_{1,4}\|\Delta_\Gamma u\|_{0,4}
		+\|u\|_{1,4}\|\Delta_\Gamma(u-\tilde{u})\|_{0,2}\|V\|_{1,4}
		+\|u-\tilde{u}\|_{1,2}\|\Delta_\Gamma \tilde{u}\|_{0,4}\|V\|_{1,4}
		\\
		&+
		\|V-V\|_{1,2+\epsilon} \|u\|_{1,\infty}\|\Delta_\Gamma u\|_{1,\frac{2+\epsilon}{1+\epsilon}}
		+\|V\|_{1,\frac{2(2+\epsilon)(2-\epsilon/2)}{\epsilon-\epsilon^2}} \|u-\tilde{u}\|_{1,2+\epsilon}\|\Delta_\Gamma u\|_{1,2-\epsilon/2}
		+\|V\|_{1,\frac{2t}{(t-1)}}\|\tilde{u}\|_{1,\frac{2t}{(t-1)}} \|\Delta_\Gamma(u-\tilde{u})\|_{1,t}
		\\
		&+
		\|V-V\|_{1,2} \|u\|_{1,4}^2
		+ \|V\|_{1,4}\|u-\tilde{u}\|_{1,2}\|u\|_{1,4}
		+ \|V\|_{1,4}\|\tilde{u}\|_{1,4}\|u-\tilde{u}\|_{1,2}\big).
	\end{align*}
\end{proposition}
\end{comment}
\begin{proposition}
	Let $\tilde{u} \in W^{1,\infty}(\Gamma)$ with $-\Delta_\Gamma\tilde{u}\in  W^{1,2-\delta}(\Gamma)$ for any $\delta>0$.
	% and $V$ as in Corollary \ref{cor:SimpleV}.
	 %(in particular $\divg V = 0$).
	Then for any $\epsilon \in (0,1),$ $p\in (1,2)$ and $q=p^*$, there is $C>0$ such that
	\begin{align*}
		|\partial_e \Jstar(0,\tilde{u}) - \partial_e \mathcal{E}(p)|
		\leq
		C \|\nabla_\Gamma V \|_{0,\infty}\Big(&
		 \|\Delta_\Gamma(u-\tilde{u})\|_{1,p} \|\nabla_\Gamma u \|_{1,q}
		\\&+\|\Delta_\Gamma(u-\tilde{u})\|_{0,2}\left(\|\Delta_\Gamma\tilde{u}\|_{0,2} +\|\Delta_\Gamma u\|_{0,2}\right)
		\\&+\|\nabla_\Gamma ( u -\tilde{u})\|_{0,\frac{2-\epsilon}{1-\epsilon}}\|\Delta_\Gamma \tilde{u}\|_{1,2-\epsilon}&
		\\&+\|\nabla_\Gamma(u-\tilde{u})\|_{1,2} \left(\|\nabla_\Gamma u \|_{1,2}+ \|\nabla_\Gamma\tilde{u})\|_{1,2}\right)
		\Big).
	\end{align*}
\end{proposition}
\begin{proof}
	This follows from the form $\partial_e\Jstar$ takes in \eqref{eq:explicitDerivativeFormula} and making use of H\"older inequalities.
\end{proof}
The particular form for the estimate above is chosen so that one may apply the error estimates of \cite{EllHer20-A} making use of a split formulation to approximate $u$ and $-\Delta_\Gamma u+u$ with linear finite elements.
There may be different estimates one wishes to show which relate to the formula of Proposition \ref{prop:RawDerivativeFormula}, for example, if one were to use a higher order discretisation of the membrane problem such as the method of \cite{LarLar17} which deals with a biharmonic problem on surfaces.
%We note that the above estimate suggests that it is sufficient to consider $V$ Lipschitz, rather than $k\geq 3$ times continuously differentiable for the approximation.

%Furthermore, we note that the formula has only dependence on the derivative of $V$.
%As such it is interesting to observe if  one was able to pick $V$ to be, in some sense, constant where  the approximation of $u$ is bad, one might hope to see convergence which is higher than the expected results of \cite{EllHer19}.

\subsection{Experiments}\label{sec:Example}
We now conduct a selection of numerical experiments.
These illustrate the formula and that the  method of difference quotients may be unreliable.
It is clear that the difference quotient will be slower - one would have to solve (at least) two algebraic systems, whereas when using the formula, a single algebraic system is solved and a functional evaluated.

%In these experiments, for calculating the derivative we replace $V:= \partial_e\chi(0,\cdot)$ with $V(e,\cdot)$ which are close (see Lemma \ref{lem:VApproximation}).
%This is admissible to do since we can, by rescaling, make the direction as small as required.
%Since directional derivatives are linear, for any $t>0$, we have that 
%$\partial_e \mathcal{E}(p) = \frac{1}{t} \partial_{te}\mathcal{E}(p)$, whence the error contribution from using $V$ instead of $V$ may be made arbitrarily small.

For all of the experiments we fix $\kappa=\sigma=R=1$.
For the optimal membrane shape, $u(p)$, we approximate it by solving a penalised finite element problem, we call this $u_h(p)$.
% as described in \cite[Section 6]{EllHer19}, we call this $u_h(p)$.
The penalisation weakly enforces the point constraints and is done in order to ease the linear algebra.
% of points which do not lie on vertives of the grid.
We solve a split system for this fourth order problem, the well-posedness and analysis of the system is given in \cite{EllHer20-A} where the error due to using a penalty formulation is shown to be well controlled.
% for this problem with point penalty.
%The vector fields $V$ which appear in the derivative are constructed to satisfy the incompressibility and point value conditions on $\Gamma$ as in Definition \ref{def:DefineV}.
%This function is then interpolated to be a function in the discrete function space.
All the experiments have been implemented under the Distributed and Unified Numerics Environment (DUNE) \cite{AlkDedKlo16,BlaBurDed16}.

%The explicit construction of $V$ which is used is to be found in \cite{Her20/EllGraHer20}.  %Thesis?  It could be incorporated into some membrane tube article.

We begin with an experiment to demonstrate the convergence of the numerical calculation of the formula.
This is done by fixing a particle configuration and refining the computational mesh.
%For the \emph{"convergence"} experiments, we fix the  configuration and refine the computational mesh observe the convergence of the computed derivative.
This experiment is then followed by some experiments where we fix the grid 
%($h = 0.0382639$)
and vary the configuration to verify that the derivative we calculate matches the a difference quotient of the energy.
In these experiments we also see that the formula is a better method than using difference quotients.

We now define the quantities which we will calculate in the numerical experiments.

%\marginpar{need more surface PDE definitions? or is this good enough?}
\begin{definition}
	Let $\Gamma_h$ be a connected, polygonal surface approximating $\Gamma$ and  $\mathcal{S}_h$ be the space of linear finite element functions on $\Gamma_h$.
	%, a surface approximating $\Gamma$.
	Given $v_h\in \mathcal{S}_h$ a finite element function, let $w_h\in \mathcal{S}_h$ satisfy
	\[
		\int_{\Gamma_h} \nabla_{\Gamma_h} v_h \cdot \nabla_{\Gamma_h} \eta_h + v_h \eta_h = \int_{\Gamma_h} w_h \eta_h
	\]
	for all $\eta_h\in \mathcal{S}_h$.
	We define
	\[
		J_h(v_h)
		:=
		\frac{1}{2}\int_{\Gamma_h} \kappa(w_h-v_h)^2 + \left(\sigma-\frac{2\kappa}{R^2}\right)|\nabla_{\Gamma_h} v_h|^2
		- \frac{2\sigma}{R^2}v_h^2,
	\]
	the discrete analogue of \eqref{eq:QuadraticSurfaceEnergy}.
	Define
	\[
		\mathcal{E}_h(p):= J_h(u_h(p)),
	\]
	the discrete analogue of Definition \ref{def:solution}, where $u_h(p)$ is the minimiser of $J_h$ over $\mathcal{S}_h$ such that $\int_{\Gamma_h}u_h(p)=0$ and $T(p)(u_h^l(p))=Z$.
	%\eqref{eq:QuadraticSurfaceEnergy}.%, where, this $u_h(p)$ is \emph{not} the $\argmin$ of $J_h$ over $\mathcal{S}_h$, but an approximation (via penalty method) of it, as described above.
		
	 Let $V_h = I_h V$, where $V$ is as in Definition \ref{def:DevineV} and $I_h\colon C(\Gamma) \to \mathcal{S}_h$ is the interpolation map.
	Then define $\mathcal{A}_h:= I ({\rm div}_{\Gamma_h} V_h) - \nabla_{\Gamma_h} V_h - \nabla_{\Gamma_h} V_h^T$ and
	\begin{align*}
		(\partial_e \Jstar)_h(v_h)
		:=
		-&\kappa\int_{\Gamma_h}\frac{1}{2}\left({\rm div}_{\Gamma_h}V\right) (v_h-w_h)^2 + \nabla_{\Gamma_h}(v_h-w_h)\cdot \mathcal{A}_h\nabla_{\Gamma_h} v_h
		\\
		+&
		\frac{1}{2}\left(\sigma - \frac{2\kappa}{R^2}\right)\int_{\Gamma_h} \nabla_{\Gamma_h} v_h \cdot \mathcal{A}_h \nabla_{\Gamma_h} v_h - \frac{\sigma}{R^2} \int_{\Gamma_h} \left({\rm div}_{\Gamma_h} V \right)v_h ^2,
%		\int_{\Gamma_h} \kappa \nabla_{\Gamma_h}(v_h-w_h)\cdot (\nabla_{\Gamma_h} V_h + \nabla_{\Gamma_h}V_h^T)\cdot \nabla_{\Gamma_h} v_h
%		\\
%		-\int_{\Gamma_h}\left(\sigma - \frac{2\kappa}{R^2}\right) \nabla_{\Gamma_h} v_h \cdot \nabla_{\Gamma_h} V_h \cdot \nabla_{\Gamma_h} v_h,
	\end{align*}
	the discrete analogue of \eqref{eq:DerivativeFullForm}.%, where again, this $u_h$ is not the $\argmin$ of $J_h$.
\end{definition}
\begin{comment}
\begin{remark}
	It is important to observe that $\partial_e \mathcal{E}_h(p)$ does not, in general, coincide with $(\partial_e\Jstar)_h(\tilde{u}(p))$, even when $\tilde{u}(p)$ is the $\argmin$ for the discrete problem.
\end{remark}
\end{comment}
\begin{comment}
\begin{proposition}
	Let	$\partial_e \mathcal{E}h$ be as above, then (?).
\end{proposition}
\end{comment}

Note that $(\partial_e J^*)_h$ is not necessarily the derivative of $\mathcal{E}_h$.
It is clear that the difference quotients we calculate will be approximations of the derivative of $\mathcal{E}_h$, should it exist, but not necessarily close to $(\partial_e J^*)_h$.

For the first three experiments we use $V(\cdot,\cdot)= \mathcal{V}(0,\cdot,\cdot)$ as in the construction in Definition \ref{CurlyV}. We take $\delta$ to be roughly $h$ so that the interpolation of $V$ has support on a  small, fixed number of vertices.
This makes the evaluation of the functional very quick.
For the remaining experiments, $V$ is constructed as in Section \ref{subsubsec:simpleChi}, where the $r$ and $\epsilon$ we use for the cut off function are taken to be $r= 0.75$ and $\epsilon = 0.15$.
% with much larger support at a slight cost of the speed of the evaluation of the functional.
%This is done because, for the particles we consider, there are insufficient grid points between some points of a particle to ensure $V$ at each point was correct using the previous method, while ensuring that for each point of a particle the parts of $V$ relating to that point has support on a sufficient number of vertices.
%In an application one might use some kind of local refinement to avoid this - therefore allowing the previous construction.
%In these cases $V$ is constructed by taking
%$V(x) = \eta(x) x \times v_p$, where $v_p$ is a non-zero vector in $\R^3$ describing the axis of rotation and the magnitude and $\eta$ a cut off function, both dependent on the problem.
%$V(x) = \eta(x) (0,0,1)^T\times x $ and $\eta(x)$ is a cut off function taken to be $1$ in $x\cdot (0,0,1)^T \geq 0.75$ and $0$ in $x\cdot(0,0,1)^T\leq 0.6$, for the implementation it is joined linearly.

For the presented convergence experiment, we do not  know the exact values of the quantities we estimate.
%, except for in highly symmetric situations.
We take the error at level $h$ to be given by the difference between the value at level $h$ and the value on the most refined grid.
That is for quantity $F_h$ and smallest grid size $h^*$, we say the error $E_h$ is given by $|F_h-F_{h^*}|$.
For two grids with size $h_1$ and $h_2$, we say the EOC of $F_h$ is given by $\log(E_{h_1}/E_{h_2})/\log(h_1/h_2)$, we will take $h_1$ and $h_2$ to be from successively refined grids.

\subsubsection{Convergence experiment}%Experiment (re)done - need to put into the tables
\label{subsec:ExperimentOne}
We begin by checking the formula and the finite element approximation. We consider $6$ particles each  consisting of a single point. The points and constraints are given by 
%\marginpar{need to think how best to display this data, possibly as the $\mathcal{G}$ for this particular experiment?}
\begin{equation*}
	\begin{split}
	X_1=& (0,0,1)^T,\,~ Z_1 = 1;
	\\
	X_3 =& (0,1,0)^T,\,~ Z_3 = 0;
	\\
	X_5 =& (1,0,0)^T,\,~ Z_5 = 0.1;
	\end{split}
	\quad\quad
	\begin{split}
	X_2 =& (0,0,-1)^T,\, Z_2 = 0;
	\\
	X_4 =& (0,-1,0)^T,\, Z_4 = 0;
	\\
	X_6 =& (-1,0,0)^T,\, Z_6 =0.
	\end{split}
\end{equation*}
%\begin{align*}
%	X_1=& (0,0,1)^T,\,~ &Z_1 =& 1;\quad~~&X_2 =& (0,0,-1)^T,\, &Z_2 =& 0;\\
%	X_3 =& (0,1,0)^T,\,~ \quad &Z_3 =& 0; \quad~~&X_4 =& (0,-1,0)^T,\, &Z_4 =& 0;\,\\
%	X_5 =& (1,0,0)^T,\,~ \quad &Z_5 =& 0.1; \quad ~~&X_6 =& (-1,0,0)^T,\,  &Z_6 =& 0.
%\end{align*}
Approximate evaluations of the  derivative in the direction $e= (1,0,0)^T \in T_{X_1}\Gamma$
are computed together with  approximations of the energy. For each finite element mesh size  $h$, we calculate
\[
	\mathcal{E}_h(0),\, \mathcal{E}_h(\theta(\delta_h)),\, \mathcal{E}_h(-\theta(\delta_h)),\, (\partial_e \Jstar)_h(u_h).
\]

 Here  $\mathcal{E}_h(\theta(\delta))$ denotes the energy  where the point $X_1$ is replaced by the point
 \[
 	X_1(\theta(\delta)) := (\sin(\theta(\delta)),0,\cos(\theta(\delta)))^T
 	\mbox{ with }
	\theta(\delta):=\arcsin\left(\frac{\delta}{\sqrt{\delta^2+(\delta-1)^2}}\right).
\]
We are then able to compute another approximation to $\partial_e \mathcal{E}(0)$ using a difference quotient
\[
	DQ_h:=\frac{(\mathcal{E}_h(\theta(\delta_h))-\mathcal{E}_h(-\theta(\delta_h)))}{( \theta(\delta_h)-\theta(-\delta_h))}
\]
of the energies.
The function $\theta$ and the values of $\delta_h$ are chosen so that $X_1(\pm\theta(\delta_h))$ lie on a vertex of the grid. The results are tabulated in Table \ref{table:ExperimentOneTable1}. Observe that the energy $\mathcal{E}_h(0)$, the difference quotient $DQ_h$ and the derivative 
$(\partial_e \Jstar)_h(u_h)$ appear to converge as $h\rightarrow 0$.
The experimental order of convergence of the derivative quantities are displayed in Table \ref{table:ExperimentOneTable2}.
\begin{comment}
with an abuse of notation,  is to denote the discrete calculation of the energy with the constraints as above, evaluated at $\delta_n$. The  
is approximated 
A For this experiment, we calculate the forcing in a given direction on a family of grids to demonstrate the convergence.
We wish to compare to the result in Proposition \ref{prop:ApproxSensitivity} when the errors are taken as in \cite[Theorem 6.1, Corollary 6.1, 6.2]{EllHer20-A}.

We also have two perturbations of the system, with a single particle being moved in opposite directions, that is to evaluate $\delta$ at a small positive and a small negative number.
This is so that we may calculate a central difference quotient.

We will use $\delta_n = 1/2^{n+1}$, where $n$ relates to $h$ by refinement level, the exact value is given in table \ref{table:ExperimentOneTable1}, 
\marginpar{Experiment went well for $\delta_n = 10^{-5}$ fixed}

% and $(\partial_e \Jstar(0,\cdot)_h)(u_h)$ is to be the discrete calculation of the forcing in the direction $e$ defined by the direction from $-t_n$ to $t_n$
\end{comment}
%\marginpar{probably need a lemma that $J_h$ and $J$ are close }
\begin{table}
\caption{Calculated quantities for experiment in Subsection \ref{subsec:ExperimentOne}}\label{table:ExperimentOneTable1}
\begin{center}
\begin{tabular}{|c|c|c|c|c|c|c|}
\hline
$h$ & $\delta_h$ & $\mathcal{E}_h(-\theta(\delta_h))$ & $\mathcal{E}_h(0)$ & $\mathcal{E}_h(\theta(\delta_h))$ & $(\partial_e \Jstar)_h(u_h)$ & $DQ_h$\\ \hline
%h         & \delta_h         & Forcing  & EnergyLeft & CentralEnergy & EnergyRight & DiffQuot     \\
%0.57735   & 0.316228  & 0.679129 & 17.9723    & 21.5133       & 19.0611     & 1.72154 \\
0.301511  & 0.25               & 16.7958              & 17.199                & 16.3577                          & -1.2195             & -1.5438  \\\hline
0.152499  & 0.125              & 15.524               & 15.5781               & 15.3318                          & -1.33257            & -1.4439  \\\hline
0.0764719 & 0.0625             & 15.0356              & 15.0309               & 14.945                           & -1.37356            & -1.40516 \\\hline
0.0382639 & 0.03125            & 14.8615              & 14.8509               & 14.8174                          & -1.38244            & -1.39168 \\\hline
0.0191355 & 0.0078125          & 14.8006              & 14.7929               & 14.7788                          & -1.38464            & -1.38720 
\\ \hline
\end{tabular}
\end{center}
\end{table}
\begin{table}
\caption{Derived quantities for experiment in Subsection \ref{subsec:ExperimentOne}}\label{table:ExperimentOneTable2}
\begin{center}\begin{tabular}{|c|c|c|c|c|}
\hline
$h$       & $E_{\partial_e\Jstar_h}$ & $E_{DQ_h}$   & EOC$_{\partial_e\Jstar_h}$ & EOC$_{DQ_h}$ \\ \hline
%0.57735   & 0.705506   & 0.304769   & --         & --            \\
0.301511  & 0.165134   & 0.156597   & --	& -- 	\\\hline%& 2.23532    & 1.02499       \\
0.152499  & 0.0520672  & 0.0567013  & 1.69327    & 1.49032       \\\hline
0.0764719 & 0.0110707  & 0.0179647  & 2.24306    & 1.66523       \\\hline
0.0382639 & 0.00219195 & 0.00448579 & 2.33893    & 2.00384       \\\hline
0.0191355 & --         & --         & --         & --                      \\ \hline
\end{tabular}
\end{center}
\end{table}
%\subsubsection{Convergence experiment 2}%Possibly a perturbation of the above, just replace the t = 1/2^n with t = 1.01/2^n
%\begin{table}
%\caption{Calculated quantities for experiment in Subsection \ref{subsec:ExperimentTwo}}\label{table:ExperimentTwoTable1}
%\end{table}
%
%\begin{table}
%\caption{Derived quantities for experiment in Subsection \ref{subsec:ExperimentTwo}}\label{table:ExperimentTwoTable2}
%\end{table}
\subsubsection{Experiment for simple particles lying on vertices}%Already done, need to write neatly, the point moving along
%\marginpar{needs a better name than calibration experiments}
\label{Subsec:CalibrationExperiment1}
For this experiment, we compute approximations of the energy and the derivative on  a sequence of configurations parametrised by the location of one point $X_1(t)$. The  configuration is defined for each $t$ by 
\begin{equation*}
\begin{split}
&X_1(t) = (\sin(\theta(t)),0,\cos(\theta(t)))^T,\,~ Z_1 = 0.1;
\\
&X_2 = (0,0,-1)^T,\,~ Z_2 = 0;
\\
&X_4 = (0,-1,0)^T,\,~ Z_4 = 0;
\end{split}
~~
\begin{split}
\hspace{1pt}
\\
X_3 =& (0,1,0)^T,\,~~~ Z_3 = 0;
\\
X_5 =& (-1,0,0)^T,\,~ Z_5 = 0,
\end{split}
\end{equation*}
%\begin{align*}
%	~~~~~~~~~~~~~~ &X_1(t) = (\sin(\theta(t)),0,\cos(\theta(t)))^T,\, Z_1 = 0.1;\,\\
%	&X_2 = (0,0,-1)^T,\, Z_2 = 0;\, ~~X_3 = (0,1,0)^T,\,~~ Z_3 = 0;\,\\
%	&X_4 = (0,-1,0)^T,\, Z_4 = 0;\, ~~X_5 = (-1,0,0)^T,\, Z_5 = 0,\,
%\end{align*}
where, $\theta$ is again defined by $\theta(t):=\arcsin\left(\frac{t}{\sqrt{t^2+(t-1)^2}}\right)$.
With this choice of $\theta$ we have that the points $X_1,...,X_5$ lie on vertices of our chosen grid for each evaluation of $t$.
%As previously described, calculations are carried out on a fixed grid. 
We calculate $\mathcal{E}_h(t)$ and $( \partial_e\Jstar)_h(u_h(t))$
% and  a central difference quotient of $\mathcal{E}_h(t)$
for $t \in \left\{\frac{m}{2^5}: m \in \mathbb{N}_0, m \leq 2^5 \right\}$.
In Figure \ref{fig:ExperimentFourFig1}, we plot $\mathcal{E}_h(t)$.
The values $( \partial_e\Jstar)_h(u_h(t))$ with the difference quotient of $\mathcal{E}_h(t)$ and also the difference between them are given in Figure \ref{fig:ExperimentFourFig2ForcingAndDQ}.
One may calculate that the relative error has a maximum of $2\%$ at the boundary and is below $1\%$ for the interior.

\begin{figure}
\begin{center}
\includegraphics[scale=0.5]{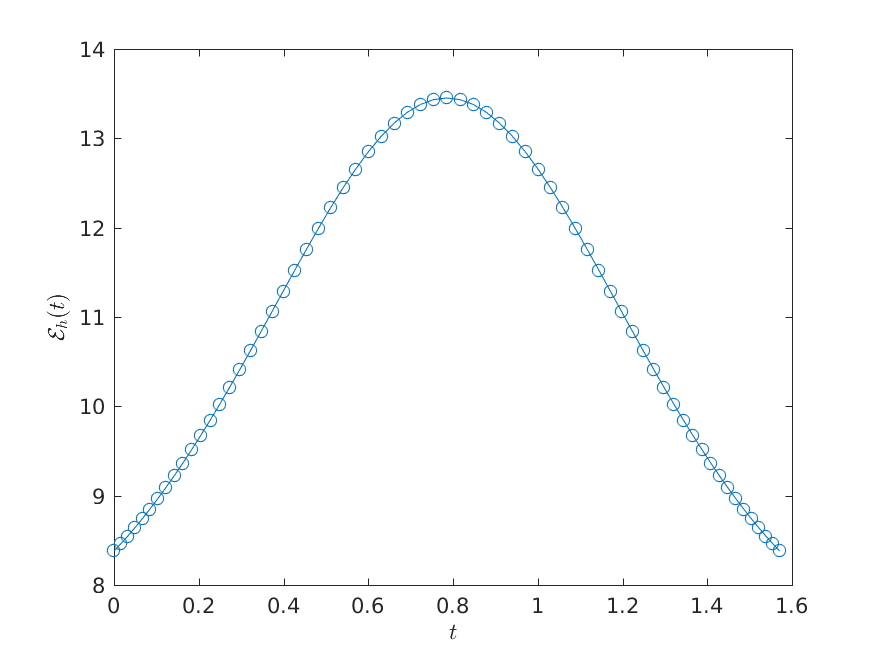}
\caption{Energy $\mathcal{E}_h(t)$ for experiment in Subsection \ref{Subsec:CalibrationExperiment1}}\label{fig:ExperimentFourFig1}
\end{center}
\end{figure}

\begin{figure}
\begin{center}
\subfigure[$(\partial_e\Jstar)_h$ and $DQ_h$]{
\includegraphics[scale=0.4]{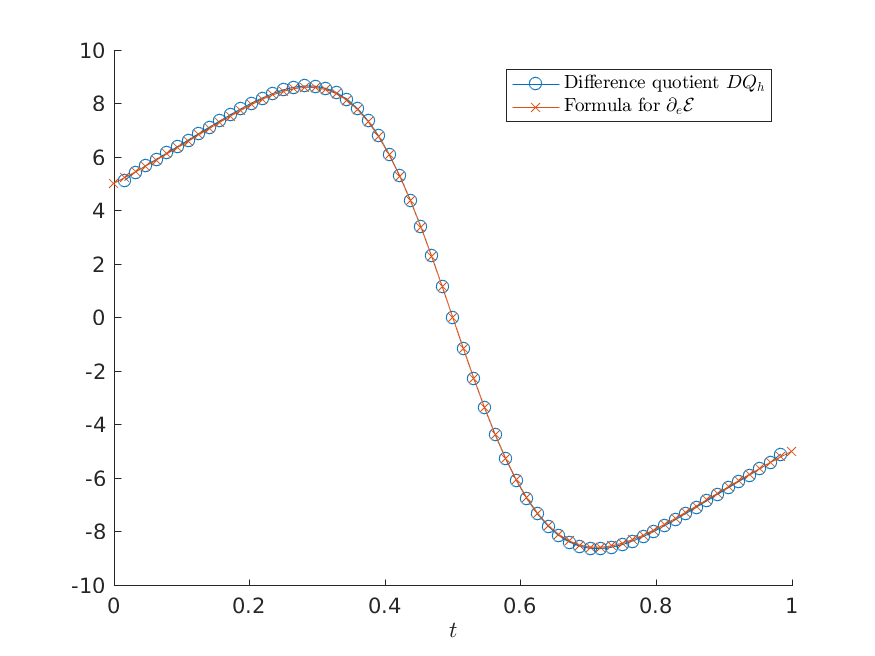}}
\subfigure[$|(\partial_e\Jstar)_h-DQ_h|$]{
\includegraphics[scale=0.4]{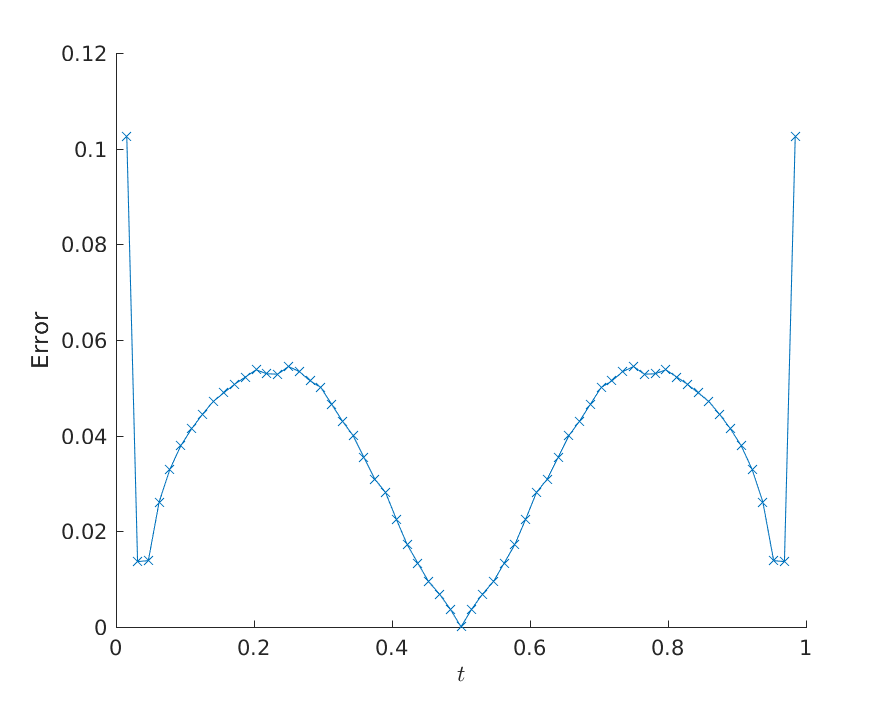}
}
\caption{Graphs of quantities from experiment in Subsection \ref{Subsec:CalibrationExperiment1}}\label{fig:ExperimentFourFig2ForcingAndDQ}
\end{center}
\end{figure}

\subsubsection{Experiment for simple particles not lying on vertices of the grid}\label{Subsec:CalibrationExperiment2}%Already done, point moving along with different time discretisation
We now provide a perturbation of the above experiment.
This experiment is to demonstrate that when the constraint points do not lie on the vertices of the grid, the difference quotient becomes a less reliable method.
%The perturbation comes in where we take $t$ to lie, previously we took $t$ such that the points $X_1(t)$ was on a vertex of the grid at every $t$ evaluated.
For this experiment we choose $t\in \{\frac{m}{100}: m \in \mathbb{N}_0, m \leq 100\}$.
We plot the same quantities as in the previous experiment.
In Figure \ref{fig:ExperimentFiveFig1}, we plot $\mathcal{E}_h(t)$, we notice it has the same characteristic shape as the previous experiment.
For Figure \ref{fig:ExperimentFiveFig2ForcingAndDQ}, we plot $( \partial_e\Jstar)_h(u_h(t))$ with the difference quotient of $\mathcal{E}_h(t)$ and also the difference between them.
We notice that here, the difference quotient does not match the formula as well as in the previous experiment.
%The resulting graphs may be found in Figures \ref{fig:ExperimentFiveFig1} and \ref{fig:ExperimentFiveFig2ForcingAndDQ}.
%We see that here the difference quotient is noisy compared to the previous experiment, with the graph of the computed derivative $(\partial_e\Jstar)_h$ is almost identical to that of the previous experiment.

\begin{figure}\begin{center}
\includegraphics[scale=0.5]{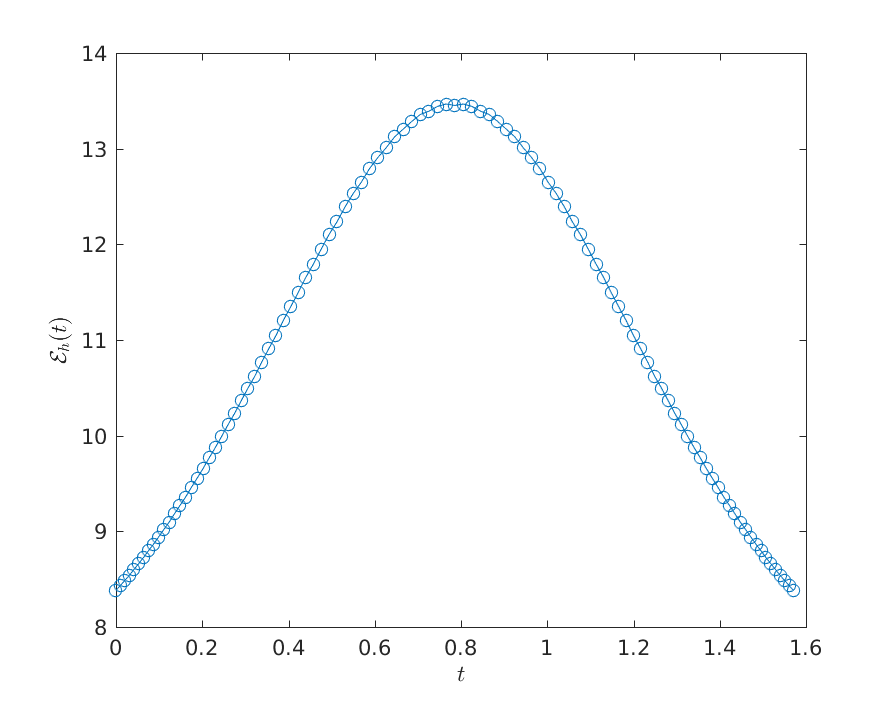}
\caption{Energy $\mathcal{E}_h(t)$ for experiment in Subsection \ref{Subsec:CalibrationExperiment2}}\label{fig:ExperimentFiveFig1}\end{center}
\end{figure}

\begin{figure}\begin{center}
\subfigure[$(\partial_e\Jstar)_h$ and $DQ_h$]{
\includegraphics[scale=0.4]{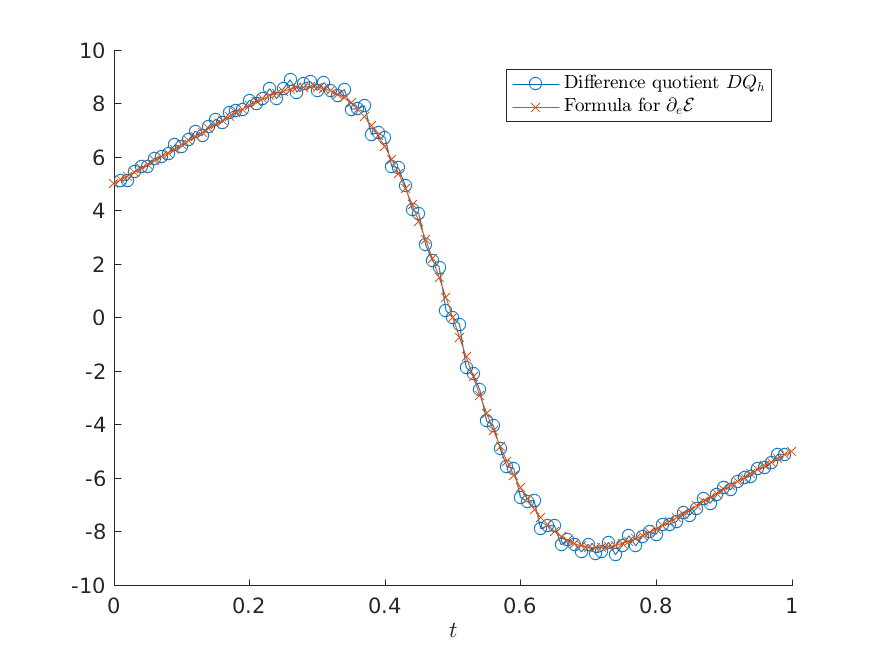}}
\subfigure[$|(\partial_e\Jstar)_h-DQ_h|$]{
\includegraphics[scale=0.4]{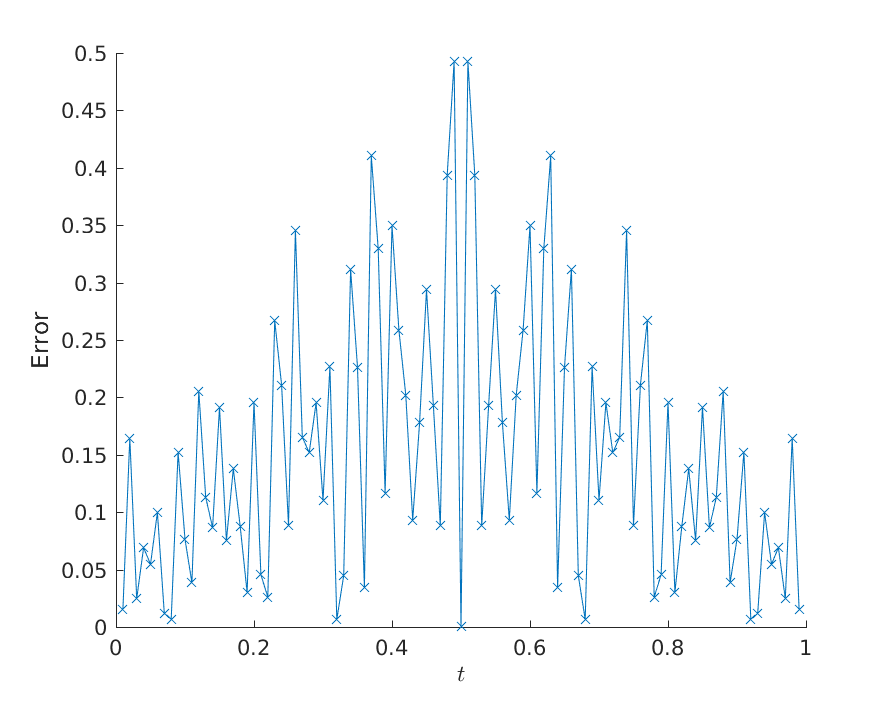}}
\caption{Graphs of quantities from experiment in Subsection \ref{Subsec:CalibrationExperiment2}}\label{fig:ExperimentFiveFig2ForcingAndDQ}
\end{center}
\end{figure}

\subsubsection{Experiment for non-trivial particles}\label{Subsec:CalibrationExperiment3}%Peanut rotation
This experiment now deals with two non-trivial particles whereby there is little chance of the points lying on vertices unless one is tailoring the grid to the points.
We will see that the difference quotients become highly unreliable.
We describe the base of the particle $\cee_1$ with centre $X_{\mathcal{G}_1} = (0,0,1)^T$.
We have that $\cee_1 = \{X_i\}_{i=1}^8$ with
\begin{equation*}
\begin{split}
	X_1 =& (0.5, 0, \sqrt{1- 0.5^2})^T,
	\\
	X_3 =& (0.25, 0.25, \sqrt{1- 0.25^2 - 0.25^2})^T,
	\\
	X_5 =& (0.25, -0.25, \sqrt{1- 0.25^2 - 0.25^2})^T,
	\\
	X_7 =& (0, 0.125, \sqrt{1 - 0.125^2})^T,
\end{split}
\quad \quad
\begin{split}
	X_2 =& (-0.5, 0, \sqrt{1- 0.5^2})^T,
	\\
	X_4 =& (-0.25, 0.25, \sqrt{1- 0.25^2 - 0.25^2})^T,
	\\
	X_6 =& (-0.25, -0.25, \sqrt{1- 0.25^2 - 0.25^2})^T,
	\\
	X_8 =& (0, -0.125, \sqrt{1- 0.125^2})^T,
\end{split}
\end{equation*}
and $(Z_1)_i = 1-\frac{1}{5}(X_i)_1^2$ for $i=1,...,8$.
%and $(Z_1)_i = 0.5$ for $i=1,...,6$ and $(Z_1)_7=(Z_1)_8 = 1$.
We let
\[
	\cee_2:= \{ x = (x_1,x_2,x_3)^T \in \Gamma : (x_1,x_3,-x_2)^T \in \cee_1\},
\]
with $(Z_2)_i = 1-\frac{1}{5}(X_i)_1^2$ for $i=1,...,8$.
%0.5$ for $i=1,...,6$ and $(Z_2)_7=(Z_2)_8 = 1$.

We consider the rotation of $\cee_1$ about the north pole, we write % in an abuse of notation,
$\cee_1(t) := \cee(0,\frac{\pi}{2} t)$.
We calculate the quantities $\mathcal{E}_h(t)$ and $(\partial_e \Jstar)_h(u_h(t))$ for $t \in \{ \frac{m}{2^4} : m \in N_0, m \leq 2^5 \}$.
We plot $\mathcal{E}(t)$ in Figure \ref{fig:ExperimentSixEnergy}.
In Figure \ref{fig:ExperimentSixDerived} we plot $(\partial_e\Jstar)_h(u_h(t))$ and the central difference quotient for $\mathcal{E}_h(t)$.
\begin{figure}\begin{center}
\includegraphics[scale=0.5]{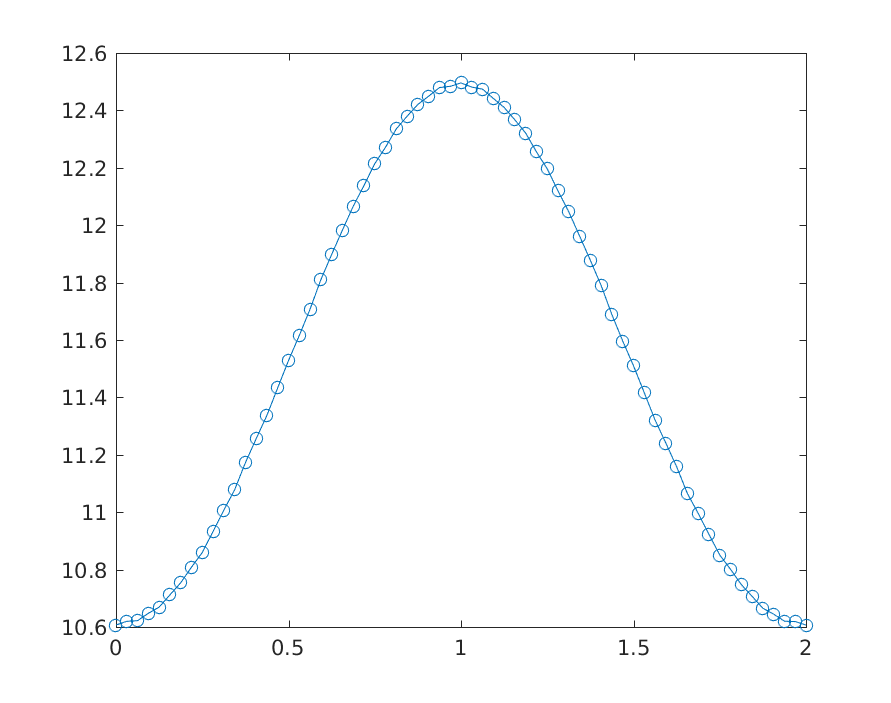}
\caption{Energy $\mathcal{E}_h(t)$ for experiment in Subsection \ref{Subsec:CalibrationExperiment3}}\label{fig:ExperimentSixEnergy}
\end{center}
\end{figure}
\begin{figure}\begin{center}
\subfigure[$(\partial_e\Jstar)_h$ and $DQ_h$]{
\includegraphics[scale=0.4]{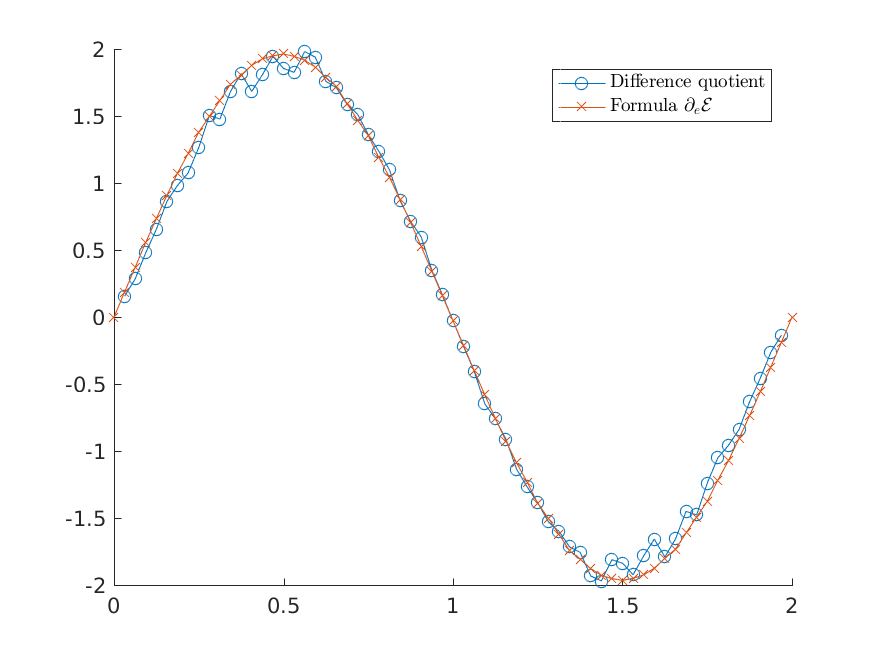}}
\subfigure[$|(\partial_e\Jstar)_h-DQ_h|$]{
\includegraphics[scale=0.4]{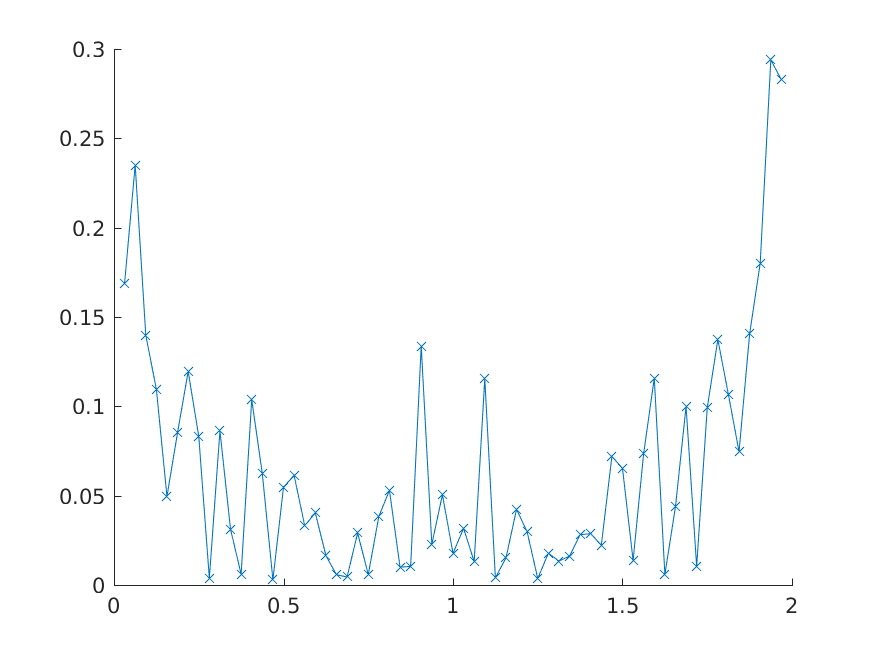}}
\caption{Graphs of quantities from experiment in experiment in Subsection \ref{Subsec:CalibrationExperiment3}}\label{fig:ExperimentSixDerived}
\end{center}
\end{figure}

\subsubsection{Experiment to observe the numerical error of a trivial system}\label{Subsec:CalibrationExperiment4}%Single peanut on top (approximation of 0)
We notice that the difference quotient in the previous experiment is extremely noisy, in this experiment, we consider a perturbation of the above experiment, where we remove $\cee_2$ so that, in light of Corollary \ref{cor:ZeroDerivative}, we are approximating zero.
The quantities from this experiment are plotted in Figure \ref{fig:ExperimentSevenDerived} where it is seen that there are moderately large perturbations from the average of the energy and the derivative is quite small, as expected.

\begin{figure}\begin{center}
\subfigure[Deviation from average of $\mathcal{E}_h(t)$]{
\includegraphics[scale=0.4]{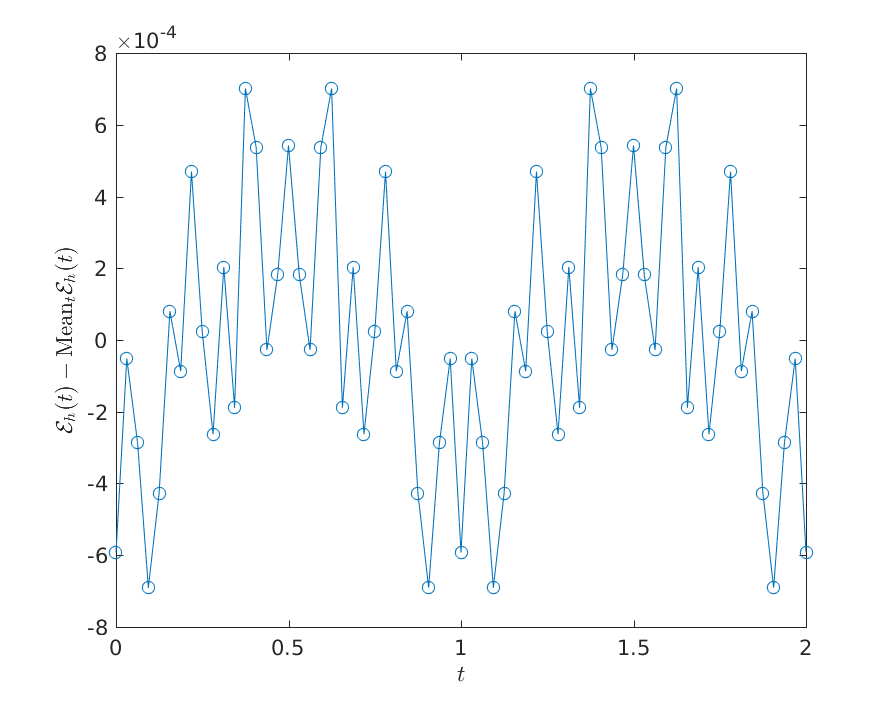}}
\subfigure[$(\partial_e \Jstar)_h$]{
\includegraphics[scale=0.4]{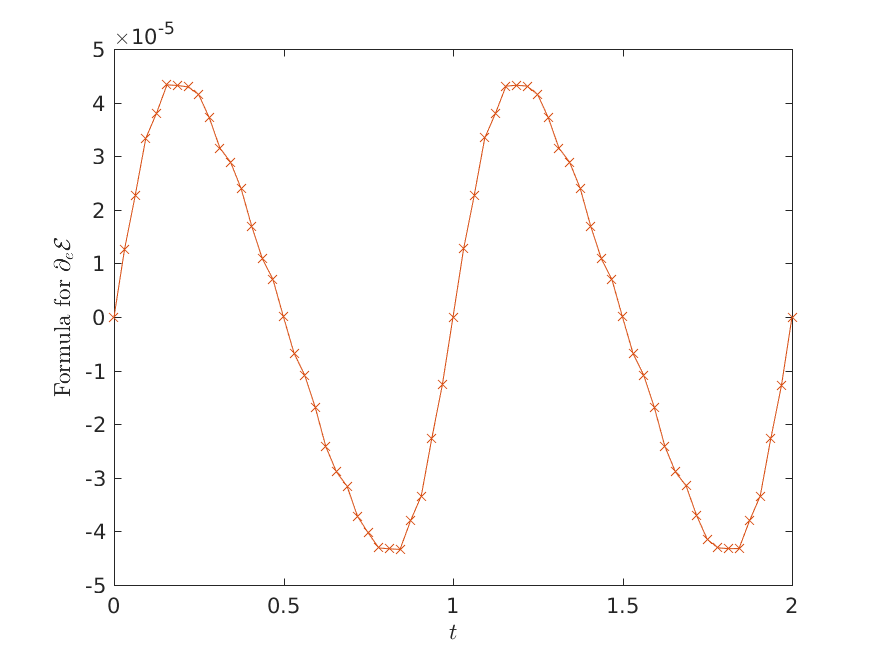}}
\caption{Graphs for experiment in Subsection \ref{Subsec:CalibrationExperiment4}}\label{fig:ExperimentSevenDerived}
\end{center}
\end{figure}

%\subsubsection{Convergence experiment 3}%Possibly do the peanut (non-trivial particle configuration) with the same as above, take $t= \pm 1/2^n$ and do convergence things
%
%\begin{table}
%\caption{Calculated quantities for experiment in Subsection \ref{subsec:ExperimentOne}}\label{table:ExperimentThreeTable1}
%\end{table}
%
%\begin{table}
%\caption{Derived quantities for experiment in Subsection \ref{subsec:ExperimentOne}}\label{table:ExperimentThreeTable2}
%\end{table}

\subsubsection{Application of formula}\label{subsec:application}
We now give the results of a numerical experiment which shows that for a perturbation of our non-trivial particles, they demonstrate a preferential orientation.
The idea of our experiment is to consider a particle based at a pole and a particle based at the equator.
We then calculate the derivative of the energy as the particle at the pole is moved towards the particle at the equator.
This experiment is then redone after rotating the particle at the pole by $\frac{\pi}{2}$.
We define the particle $\cee_1 = \{X_i\}_{i=1}^8$ by
\begin{equation*}
\begin{split}
	X_1 =& (0.3, 0, \sqrt{1- 0.3^2})^T,
	\\
	X_3 =& (0.15, 0.15, \sqrt{1- 0.15^2 - 0.15^2})^T,
	\\
	X_5 =& (0.15, -0.15, \sqrt{1- 0.15^2 - 0.15^2})^T,
	\\
	X_7 =& (0, 0.075, \sqrt{1 - 0.075^2})^T,
\end{split}
\quad \quad
\begin{split}
	X_2 =& (-0.3, 0, \sqrt{1- 0.3^2})^T,
	\\
	X_4 =& (-0.15, 0.15, \sqrt{1- 0.15^2 - 0.15^2})^T,
	\\
	X_6 =& (-0.15, -0.15, \sqrt{1- 0.15^2 - 0.15^2})^T,
	\\
	X_8 =& (0, -0.075, \sqrt{1- 0.075^2})^T,
\end{split}
\end{equation*}
and $(Z_1)_i = 1-0.9(X_i)_1^2$ for $i=1,...,8$.
We give this centre $X_{\mathcal{G}_1}:= (0,0,1)^T$.
We define $\cee_2$ by
\[
	\cee_2 := \{x = (x_1,x_2,x_3)^T \in \Gamma : (x_1,x_3,-x_2)^T \in \cee_1\},
\]
with $(Z_2)_i = 1-10(X_i)_1^2$ for $i=1,...,8$ and centre .

We calculate the derivative at $0 \in \prod_{i=1}^2 (\R \times T_{X_{\mathcal{G}_i}}\Gamma)$ in direction $e=(0,\tau,0,0)$, where $\tau=(0,1,0)^T \in T_{X_{\mathcal{G}_1}}$ represents the translation of $\cee_1$ in the direction $\tau$.

We then calculate the derivative at $p:=(\frac{\pi}{2},0,0,0)$ in the same direction $e$.

We find that
\[
	(\partial_eJ^*)_h(0) \approx -10.6729 \quad \mbox{and} \quad (\partial_eJ^*)_h(p) \approx \mbox{18.5636}.
\]%3.00261 if we rotate both particles!
This shows that the orientation affects whether the particles are attracted to each other, with one orientation being repulsive and the other attractive.
In Figure \ref{fig:MembranesForApplciation} we give the numerical approximations for membranes $u(0)$ and $u(p)$.
\begin{figure}
\begin{center}
\includegraphics[width=.45\linewidth]{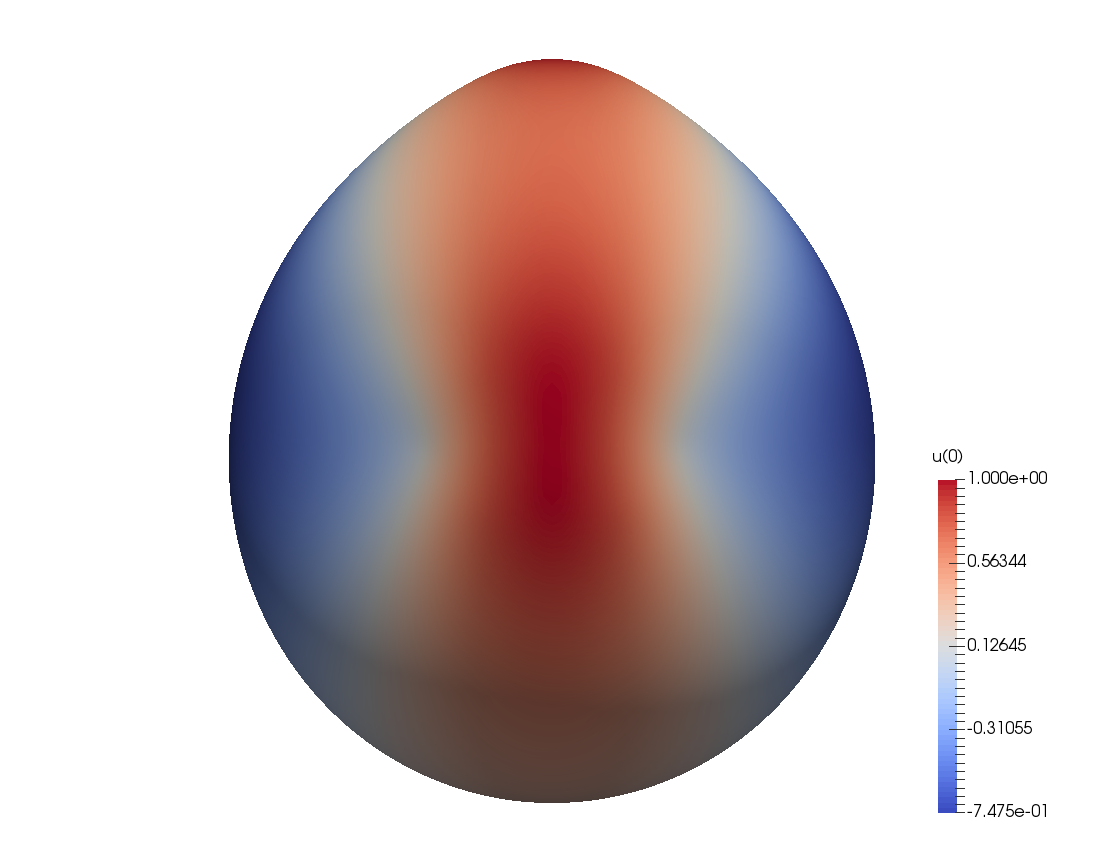}
\includegraphics[width=.45\linewidth]{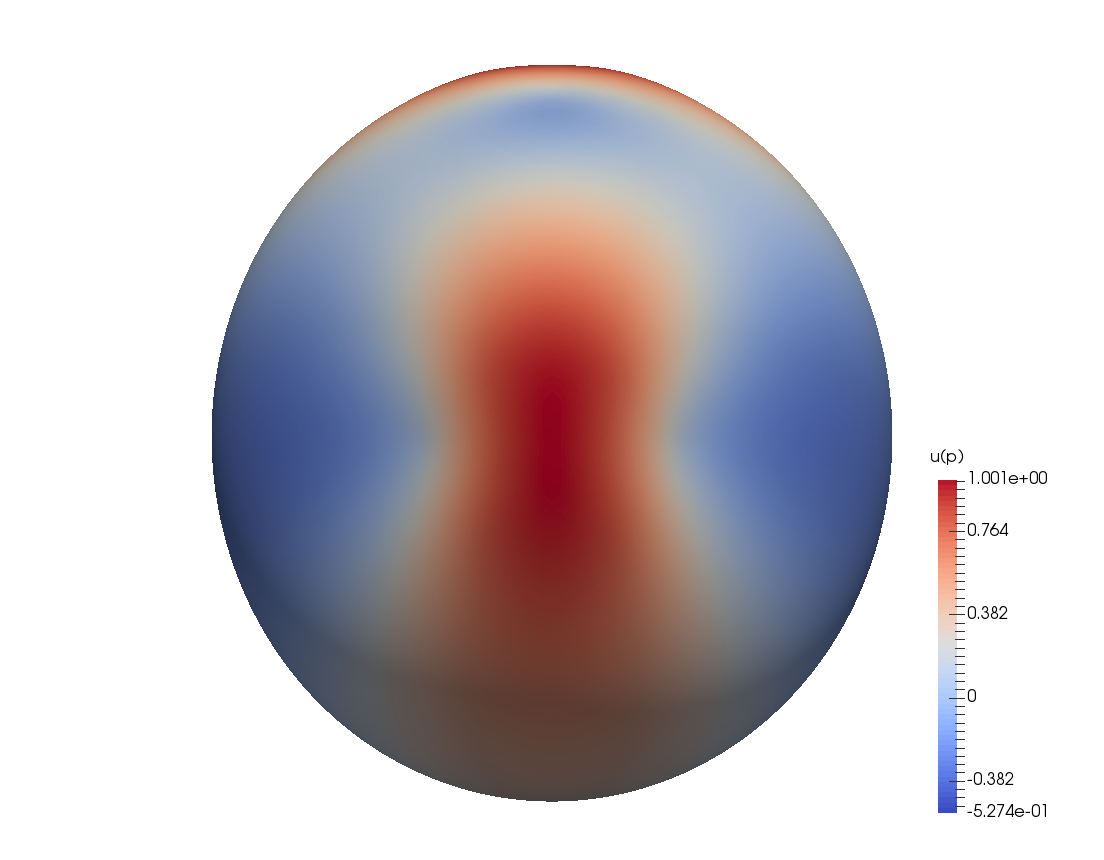}
\end{center}
\caption{The rescaled graphs of the membranes from experiment in Subsection \ref{subsec:application}, left $0.1\,u(0)$, right $0.1\,u(p)$, both with $(0,1,0)^T$ coming out of the page and $(0,0,1)^T$ pointing up.
The colours represent the magnitude of the deformation.}\label{fig:MembranesForApplciation}
\end{figure}

%\subsection{A simple gradient flow experiment}
%Might be pretty trivial to put in an experiment where we have 4 particles on equator and one on a pole and we see what the optimal orientation for these is

\section{Conclusion}
In this article we have shown the differentiability of $\mathcal{E}(p)$, the membrane mediated interaction energy for a near spherical membrane with particles attached at points which depend smoothly on $p$.
Further to showing the differentiability, we have given an explicit formula to calculate the derivative and give numerical examples which demonstrate that this formula would appear to be more robust than a difference quotient approach.

It would be of interest to extend this analysis for particles which are able to move more generally, tilting and moving out from the surface.
Furthermore it is desirable to consider the problem for inequality constraints on the 'interior' of a particle.
Finally, one could analyse higher order derivatives of the energy so that one could determine stability of a given configuration.

\section*{Acknowledgements}
The work of CME was partially supported by the Royal Society via a Wolfson Research Merit Award. 
The research of PJH was funded by the Engineering and Physical Sciences Research Council grant  EP/H023364/1 under the MASDOC centre for doctoral training at the University of Warwick.
%\end{acknowledgements}

\appendix
\newcommand{\Chi}{\mathrm{X}}
\section{The pullback to a reference domain}\label{sec:Reformulation}

We give some general results on the calculation of the composition of pullbacks and derivatives, where we consider that the image and domain of the diffeomorphism need not be the same.
As we are working with different surfaces, we will need to make clear to which surface geometric quantities belong to, this is done with a superscript of the surface, e.g. $H^{\Gamma_1}$ is the mean curvature of $\Gamma_1$ and $H^{\g}$ the mean curvature of $\g$.
%Given structure preserving bijections between the admissible sets, we now use these bijections to find a transformed energy which is posed over a reference space.
Consider the case of $\Gamma_0$ and $\Gamma_1$ being $C^k$, compact surfaces, with $\Chi\colon \Gamma_0 \to \Gamma_1$ a $C^k$-diffeomorphism, where we require $k\geq 2$.%, eventually we will have that $\mathcal \Chi$ depends smoothly on $p \in \Lambda^\circ$ and a direction in which we wish to move the constraints.

Given some function $u\colon \Gamma_1 \to \mathbb{R}$ we wish to obtain expressions for $(\nabla_{\Gamma_1} u) \circ \Chi$ and $(\underline{D}^2_{\Gamma_1} u) \circ \Chi$.
The first part of this is developed in \cite{ChuDjuEll19}, where also the trace of the second quantity, the Laplace-Beltrami, is calculated.
Although for the model we consider in this work, the surface Hessian is not required, we compute it for completion as it may arise in other elastic type models, where the  Hessian regularly arises.
%, where one may be interested in this kind of calculation to determine optimality of point locations as in \cite{ButNaz11}.
We choose to do this in an method which avoids integration by parts so that surfaces with boundary may be considered.

\begin{lemma}\label{lem:FirstOrderTrans}
	Let %$u \colon \Gamma\to \mathbb{R}$ satisfy 
	$u \in H^1(\Gamma_1)$,
	%Then we calculate/Then it holds,
	%then the pull-back of the tangential gradient onto the surface $\Gamma_0$ is given by
	then $u \circ \Chi \in H^1(\Gamma_0)$ and% it holds
	\[
		(\nabla_{\Gamma_1} u) \circ \Chi
		=
		\left( \nabla_{\Gamma_0}\Chi + \nu^{\Gamma_1}\circ \Chi \otimes \nu^{\Gamma_0} \right)^{-T} \nabla_{\Gamma_0} (u\circ \Chi)
		=
		\nabla_{\Gamma_0}\Chi G_{\Gamma_0}^{-1} \nabla_{\Gamma_0}(u\circ \Chi),
	\]
	where $G_{\Gamma_0}:= \nabla_{\Gamma_0}\Chi^T\nabla_{\Gamma_0}\Chi + \nu^{\Gamma_0} \otimes \nu^{\Gamma_0}$.
\end{lemma}

The proof is shown in Lemma 3.2 of \cite{ChuDjuEll19}.
We write $B := \nabla_{\Gamma_0} \Chi + \nu^{\Gamma_1} \circ \Chi \otimes \nu^{\Gamma_0}$, which satisfies
\[
	B^T B = G_{\Gamma_0}.%Is this worth while writing?
\]
This gives a simpler form of the above lemma,
\[
	(\nabla_{\Gamma_1} u) \circ \Chi = B^{-T} \nabla_{\g} (u \circ \Chi).
\]

\begin{lemma}\label{lem:SecondOrderTrans}
	Let %$u\colon \Gamma \to \mathbb{R}$ satisfy 
	$u \in H^2(\Gamma_1)$, 
	then $u \circ \Chi \in H^2(\Gamma_0)$ and for $i,j=1,...,n+1$ %written like this so that mistakes aren't made for the symmetry the matrix should have!
%\begin{align*}
%	(\underline{D}_i^{\Gamma} \nabla_{\Gamma} u) \circ \Chi
%	=&
%	\frac{1}{\det(B)}\sum_{k=1}^{n+1} \underline{D}_k^{\Gamma_0} \left( \det(B) B^{-T} \nabla_{\Gamma_0} (u\circ \Chi) \left(B^{-T} e_k \right)_i\right)
%	\\&+ B^{-T} \nabla_{\Gamma_0} (u\circ \Chi) (H^\Gamma \circ \Chi - H^{\Gamma_0} )\nu_i^\Gamma \circ \Chi.
%\end{align*}
\begin{align*}
	(\D{i}^{\Gamma_1} \D{j}^{\Gamma_1} u) \circ \Chi
	 =&
	 \frac{1}{b}{\rm div}_\g \left( b B^{-1} (B^{-T}\nabla_\g \hat{u} )_j\right)_i
	 \\
	 &+
	 (H^{\Gamma_1}\circ \Chi - H^\g)(\nu_i^{\Gamma_1} \circ \Chi)(B^{-T} \nabla_\g \hat{u})_j,
\end{align*}
where $b = \det(B),\, b_{ij} = B_{ij}$ and $b^{ij} = \left(B^{-1}\right)_{ij}$.
%\marginpar{Not sure if this is meant to be div, it is $\D{l}^\g \left( b^{li} \mbox{scalar} \right)$}
\end{lemma}
\begin{proof}
We write $\hat{u}:= u \circ \Chi$ and where indices are repeated in a product, summation is assumed.
We now make use of the preceding lemma to obtain,
\begin{align*}
	\underline{D}_i^{\Gamma_1} \underline{D}^{\Gamma_1}_j u \circ \Chi
	=&
%	\left( B^{-T}\nabla_\Gamma \left(\D{j}u \circ \Chi \right) \right)_i
%	\\
%	=&
%	\left( B^{-T}\nabla_\Gamma \left( B^{-T} \nabla_\Gamma (\hat{u}) \right)_j \right)_i
%	\\
%	=&
	b^{li}\D{l}^{\Gamma_0}\left( b^{kj} \D{k}^{\Gamma_0} \hat{u} \right).
\end{align*}
We now put this into something similar to a divergence form,
\begin{align*}
	\underline{D}_i^{\Gamma_1} \underline{D}_j^{\Gamma_1} u \circ \Chi
	=&
	%\frac{1}{b} \D{l}^{\Gamma_0}\left( b b^{li} b^{kj}\D{k}^{\Gamma_0}\hat{u} \right) +  b^{li}\D{l}^{\Gamma_0}\left( b^{kj}\D{k}^{\Gamma_0} \hat{u} \right)
	%- \frac{1}{b} \D{l}^{\Gamma_0}\left( b b^{li}\right) b^{kj}\D{k}^{\Gamma_0}\hat{u} - b^{li}\D{l}^{\Gamma_0} \left( b^{kj}\D{k}^{\Gamma_0}\hat{u} \right)
	%\\
	\frac{1}{b} \D{l}^{\Gamma_0}\left( b b^{li} b^{kj}\D{k}^{\Gamma_0}\hat{u} \right) 
	- \frac{1}{b} \D{l}^{\Gamma_0}\left( b\right) b^{li} b^{kj}\D{k}^{\Gamma_0}\hat{u} - \D{l}^{\Gamma_0}\left( b^{li}\right) b^{kj}\D{k}^{\Gamma_0}\hat{u}.
\end{align*}
In \cite{ChuDjuEll19}, it is calculated
\[
	\D{l}^{\Gamma_0}b^{li} = -b^{lm}\D{l}^{\Gamma_0}b_{mf}b^{fi},~~ \frac{1}{b}\D{l}^{\Gamma_0}b = b^{fm} \D{l}^{\Gamma_0}b_{mf},
\]
inserting these into the above gives,
\begin{equation*}
	\underline{D}_i^{\Gamma_1} \underline{D}_j^{\Gamma_1} u \circ \Chi
	=
	\frac{1}{b} \D{l}^{\Gamma_0}\left( b b^{li} b^{kj}\D{k}^{\Gamma_0}\hat{u} \right) 
	- b^{fm}\D{l}^{\Gamma_0}b_{mf} b^{li} b^{kj}\D{k}^{\Gamma_0}\hat{u} + b^{lm}\D{l}^{\Gamma_0} b_{mf} b^{fi} b^{kj}\D{k}^{\Gamma_0}\hat{u}.
\end{equation*}
Since we are summing over $f,\,k,\,l$ and $m$ in the above, it is possible to swap the indices, in particular we swap $f$ and $l$ in the second term.
We now consider the terms\begin{equation}\label{eq:HessianIntermediateEquation}
\begin{split}
	b^{lm}\D{l}^{\Gamma_0} b_{mf} b^{fi} b^{kj}\D{k}^{\Gamma_0}\hat{u} -& b^{lm}\D{f}^{\Gamma_0}b_{ml} b^{fi} b^{kj}\D{k}^{\Gamma_0}\hat{u}
	\\&=
	b^{lm}b^{fi} \left( b^{kj}\D{k}^{\Gamma_0} \hat{u}\right) \left( \D{l}^{\Gamma_0} b_{mf} - \D{f}^{\Gamma_0} b_{ml} \right).
	\end{split}
\end{equation}
In order to simplify this, we will use the definition of $B$ and swap the order of derivatives.
As in \cite{ChuDjuEll19}, one calculates
\begin{align*}
	\D{l}^{\Gamma_0}b_{mf} - \D{f}^{\Gamma_0}b_{ml}
	=&
	\left( \D{l}^{\Gamma_0}(\nu_m^{\Gamma_1} \circ \Chi ) - \left(\mathcal{H}^\g \nabla_\g \Chi_m \right)_l \right)\nu_f^\g
	\\
	&+
	\left(\left(\mathcal{H}^\g \nabla_\g \Chi_m \right)_f  - \D{f}^{\Gamma_0}(\nu_m^{\Gamma_1} \circ \Chi) \right)\nu_l^\g.
\end{align*}
We now use this to simplify \eqref{eq:HessianIntermediateEquation}.
We will make use of the relation $b^{ki}\nu_k^\g = \nu_i^{\Gamma_1} \circ \Chi$.
We calculate each part,
\begin{align*}
b^{lm}b^{fi}\D{l}^\g (\nu_m^{\Gamma_1}\circ \Chi)\nu_f^\g
=&
(\nu_i^{\Gamma_1}\circ \Chi) (B^{-T}(\nabla_{\Gamma_1}\nu_m) \circ \Chi)_m
\\
=&
(H^{\Gamma_1} \nu_i^{\Gamma_1}) \circ \Chi,
\end{align*}
\begin{align*}
b^{lm}b^{fi} (\mathcal{H}^\g \nabla_\g \Chi_m)_l \nu_f^\g
=&
b^{lm}b^{fi}\mathcal{H}^\g_{lk}\D{k}^\g \Chi_m \nu_f^\g
\\
=&
b^{lm}b^{fi}\mathcal{H}^\g_{lk}b_{mk} \nu_f^\g
\\
=&
H^\g (\nu^{\Gamma_1}_i \circ \Chi),
\end{align*}
\begin{align*}
b^{lm}b^{fi} (\mathcal{H}^\g \nabla_\g \Chi_m)_f \nu_l^\g
=&
b^{lm}b^{fi} \mathcal{H}^\g_{fk} \D{k}^\g \Chi_m \nu_l^\g
\\
=&
b^{lm}b^{fi} \mathcal{H}^\g_{fk} b_{mk} \nu_l^\g
\\
=&
b^{fi} \mathcal{H}^\g_{fl} \nu_l^\g
=
0,
\end{align*}
\begin{align*}
b^{lm}b^{fi} (\nu^{\Gamma_1}_m\circ \Chi)\nu_l^{\Gamma_1}
=&
(\nu_m^{\Gamma_1}\circ \Chi)(B^{-T} \nabla_\g (\nu^{\Gamma_1} \circ \Chi) )_i
\\
=&
(\nu_m^{\Gamma_1}\circ \Chi)\mathcal{H}^{\Gamma_1}_{mi}\circ \Chi
=0.
\end{align*}
This then gives% that
\[
	b^{lm}b^{fi} \left( b^{kj}\D{k}^{\Gamma_0} \hat{u}\right) \left( \D{l}^{\Gamma_0} b_{mf} - \D{f}^{\Gamma_0} b_{ml} \right)
	=
	((H^{\Gamma_1} \circ \Chi) -H^\g) (\nu_i^{\Gamma_1}\circ \Chi) (B^{-T} \nabla_\g \hat{u})_j,
\]
which completes the result.
\end{proof}

\begin{remark}
	By taking the trace of $\underline{D}_{\Gamma_1}^2 u \circ \Chi$, one obtains
	\[
		(\Delta_{\Gamma_1} u) \circ \Chi
		=
		\frac{1}{b} \Div_\g ( b G_\g^{-1} \nabla_{\Gamma_0} (u\circ \Chi)).
	\]
	%which is consistent with that shown in \cite{ChuDjuEll19}.
\end{remark}

%We now reformulate the energy \eqref{eq:QuadraticSurfaceEnergy}
%\[
%	J(u) := \int_{\Gamma} \kappa(\Delta_\Gamma u)^2 + \left( \sigma - \frac{2\kappa}{R^2} \right)|\nabla_\Gamma u|^2 - \frac{2\sigma}{R^2}u^2
%\]
%over the reference function space.% for the reference particle configuration.

\section{Implicit function theorem}% and Whitney's extension theorem}
\label{sec:ImplicitFunctionTheorem}
We give the version of the implicit function theorem we use in Theorem \ref{thm:applicationOfIFT}.
The result is taken from \cite[Theorem 7.13-1]{Cia13}.
\begin{theorem}\label{Thm:ImplictFnThm}
	Let $\mathcal X$ be a normed vector space, $\mathcal Y$ and $\mathcal Z$ Banach spaces with $\Omega\subset \mathcal X\times \mathcal Y$ open with $(a,b) \in \Omega$.
	Let $f\in C(\Omega;\mathcal Z)$ with $f(a,b)=0$,
	$\frac{\partial f}{\partial y}(x,y) \in \mathcal{L}(\mathcal Y;\mathcal Z)$ exists at all points $(x,y) \in \Omega$ and $\frac{\partial f}{\partial y} \in C(\Omega;\mathcal{L}(\mathcal Y;\mathcal Z)),$
	$\frac{\partial f}{\partial y}(a,b)$ is a bijection, so that
	$\left(\frac{\partial f}{\partial y}(a,b)\right)^{-1}\in \mathcal{L}(\mathcal Z;\mathcal Y)$.
	\begin{enumerate}
	\item
	Then there is an open neighbourhood $V$ of $a$ in $\mathcal X$, a neighbourhood $W$ of $b$ in $\mathcal Y$ and $g \in C(V;W)$ such that $V\times W\subset \Omega$ and 
	$\{(x,y)\in V\times W: f(x,y)=0\}=\{(x,y) \in V\times W : y= g(x)\}$.
	\item
	Assume in addition that $f$ is differentiable at $(a,b) \in \Omega$.
	Then $g$ is differentiable at $a$ and
	\[
		g'(a) = -\left(\frac{\partial f}{\partial y}(a,b)\right)^{-1} \frac{\partial f}{\partial x}(a,b) \in \mathcal{L}(\mathcal X;\mathcal Y).
	\]
	\item
	Assume in addition that $f \in C^k(\Omega;\mathcal Z)$ for some $k\geq 1$.
	Then there is an open neighbourhood $\tilde{V}\subset V$ of $a$ in $\mathcal X$ and neighbourhood $\tilde{W}\subset W$ of $b$ in $\mathcal Y$ such that
	$\frac{\partial f}{\partial y}(x,y)\in \mathcal{L}(\mathcal Y;\mathcal Z)$ is a bijection, so that $\left(\frac{\partial f}{\partial y}(x,y)\right)^{-1}\in \mathcal{L}(\mathcal Z;\mathcal Y)$ at each $(x,y) \in \tilde{V}\times \tilde W$,
	$g\in C^k(\tilde{V};\mathcal Y)$,
	$g'(x) = - \left(\frac{\partial f}{\partial y}(x,g(x)) \right)^{-1}\frac{\partial f}{\partial x}(x,g(x)) \in \mathcal{L}(\mathcal X;\mathcal Y)$ for each $x \in \tilde{V}$.
	\end{enumerate}
\end{theorem}

\begin{comment}
We also require the following extension theorem due to Whitney \cite[Theorem I]{Whi34}.
\begin{theorem}\label{thm:Whitney}
	Let $A \subset \R^n$ be closed, $m \in \mathbb{N}\cup\{\infty\}$, $f_\alpha\colon A \to \R$ for all multi-indices $\alpha\in \mathbb{N}^n$ with $|\alpha|\leq m$.
	Suppose that for all multi-indices $\alpha$ with $|\alpha|\leq m$ and $x,y\in A$ hold
	\[
		f_\alpha(y) = \sum_{\beta \leq m - |\alpha|} \frac{f_{\beta+\alpha} (x)}{\beta!} (y-x)^\beta + R_\alpha (x,y)
	\]
	where $R_\alpha A\times A\to \R$ is a remainder function such that for all $x_0 \in A$ and all $\epsilon>0$ there is $\delta>0$ such that
	\[
		\forall x,y \in A: \|x-x_0\| < \delta ~ \& ~ \|y-x_0\| < \delta \implies 	|R(x,y)| \leq \|x-y\|^{m-|\alpha|}\epsilon.
	\]
	Then there is a function $F\in C^m(\R^n)$ such that $\partial_\alpha F = f_\alpha$ on $A$ for all $\alpha$ with $|\alpha| \leq m$	.
\end{theorem}
\end{comment}

\section{Elliptic regularity}\label{app:EllReg}
%Here we show that $u \in H^2(\Gamma)$, the solution of problem \ref{prob:SingleParticleEnergyMinimisation} satisfies $u\in W^{3,p}(\Gamma)$ for $p <2$.
We first show, for arbitrary surfaces, that $\Delta_\Gamma u \in W^{1,p}(\Gamma)$ for $p\leq 2$ gives $u \in W^{3,p}(\Gamma)$.

\begin{proposition}\label{app:Prop:EllReg}
	Suppose $u \in H^1(\Gamma)$ with $\Delta_\Gamma u \in W^{1,p}(\Gamma)$ for some $p\in(1,2]$ and $\Gamma$ is $C^3$, then there is a $C>0$ independent of $u$ such that for each $i,j=1,2,3$,
	\[
		\|\D{i}\D{j} u \|_{1,p} \leq C \left( \|\D{j} \Delta_\Gamma u\|_{0,p} + \|\Delta_\Gamma u\|_{0,2} + \|\nabla_\Gamma u\|_{0,2} \right).
	\]
\end{proposition}
\begin{proof}
	We make use of the following inf-sup condition, shown in \cite{EllFriHob19}:
	\[
		\exists \gamma>0 : \gamma\|\xi\|_{1,p} \leq \sup_{\eta \in W^{1,q}(\Gamma)} \frac{\int_\Gamma \nabla_\Gamma \eta \cdot \nabla_\Gamma \xi + \eta \xi}{\|\eta\|_{1,q}}\quad\forall \xi \in W^{1,p}(\Gamma).
	\]
	By the fact that $\Gamma$ has finite measure, it holds that $\|\D{i}\D{j} u \|_{0,p} \leq C \|\D{i}\D{j}u\|_{0,2}$ which we know is controlled by $\|\Delta_\Gamma u \|_{0,2} + \sqrt{\|\mathcal{H}H - 2\mathcal{H}^2\|_{0,\infty}}\|\nabla_\Gamma u\|_{0,2}$, \cite{DziEll13}.
	It is then sufficient to show that $\int_\Gamma \nabla_\Gamma \D{i}\D{j} u \cdot\nabla_\Gamma \eta$ is bounded appropriately.
	One may calculate
	\begin{comment}
	\[
		\int_\Gamma \D{k}\D{i}\D{j} u \D{k}\eta
		=
		\int_\Gamma \left(\D{i}\D{k}\D{j} u - \left(\mathcal{H}\nabla_\Gamma\D{j}u \right)_k\nu_i \right)\D{k}\eta,
	\]
	\[
		\int_\Gamma \D{i}\D{k}\D{j} u \D{k}\eta
		=
		-\int_{\Gamma} \D{k}\D{j}u \D{i}\D{k}\eta
		=
		-\int_\Gamma \D{k}\D{j}u \left( \D{k}\D{i} \eta + \left(\mathcal{H}\nabla_\Gamma \eta \right)_k \nu_i \right)
	\]
	\[
		-\int_\Gamma \D{k}\D{j}u \D{k}\D{i}\eta
		=
		\int_\Gamma \left(\D{k}\D{j}\D{k} u - \D{k}\left[ \left(\mathcal{H}\nabla_\Gamma u\right)_k \nu_j \right]\right)\D{i}\eta
	\]
	\[
		\int_\Gamma \D{k}\D{j}\D{k} u \D{i}\eta
		=
		\int_\Gamma \left( \D{j}\D{k}\D{k} u + \left(\mathcal{H}\nabla_\Gamma\D{k}u \right)_j\nu_k - \left(\mathcal{H}\nabla_\Gamma \D{k}u\right)_k \nu_j \right)\D{i}\eta,
	\]
	\end{comment}
	\begin{align*}
		\int_\Gamma \nabla_\Gamma \D{i}\D{j} u \cdot\nabla_\Gamma \eta
		=
		\int_\Gamma &\D{j} \Delta_\Gamma u \D{i} \eta
		\\&+
		\left( \left( \mathcal{H}\nabla_\Gamma \D{k} u\right)_j\nu_k- \left(\mathcal{H}\nabla_\Gamma \D{k} u \right)_k\nu_i  -\D{k}\left[ \left(\mathcal{H}\nabla_\Gamma u \right)_k\nu_j\right]\right) \D{i}\eta
		\\
		&-\D{k}\D{j}u \left(\mathcal{H}\nabla_\Gamma \eta\right)_k \nu_i - \left( \mathcal{H}\nabla_\Gamma \D{j} u\right)_k\nu_i \D{k}\eta.
	\end{align*}
	This follows from repeatedly applying integration by parts and swapping the order of derivatives.
	Applying H\"older's inequality, the result immediately follows.
\end{proof}

\begin{proposition}
	Let $u \in H^2(\Gamma)$ be the unique solution of Problem \ref{prob:SingleParticleEnergyMinimisation}, then it holds that for any $p<2$,  $u\in W^{3,p}(\Gamma)$.
\end{proposition}
\begin{proof}
By \cite[Theorem 2.34]{ErnGue04} and the arguments presented in \cite[Section 5]{EllHer20-A}, it is clear that 
	there is $\bar{p}\in \R$ and $\lambda \in \R^K$ such that
	\[		a(u,v) + \bar{p}\int_\Gamma v + \lambda \cdot v|_{\mathcal{C}} =0~~ \forall v \in H^2(\Gamma).
	\]
	Let $\eta:= -\Delta_\Gamma u - \frac{2}{R^2}u \in L^2(\Gamma)$, then for any $v \in H^2(\Gamma)$,
	\[
		a(u,v) = \int_\Gamma (-\kappa \Delta_\Gamma v + \sigma v)\eta = -\lambda \cdot v|_{\mathcal{C}} -\bar{p}\int_\Gamma v.
	\]
	Let $\phi \in C^\infty(\Gamma)$ and consider the inverse Laplace type map $G\colon L^2(\Gamma) \to H^2(\Gamma)$ such that $G\colon \phi \mapsto v$ where $-\kappa\Delta_\Gamma v + \sigma v = \phi$.
	Via a local argument, it may be seen that for any $q>2$, $\|v\|_{0,\infty} \leq C\|\phi\|_{-1,q}$ \cite{Nec11}.
	Hence
	\begin{align*}
		\langle \phi,\eta \rangle
		&=
		\int_\Gamma \phi \eta
		=
		\int_\Gamma (-\kappa\Delta_\Gamma v + \sigma v)\eta
		\\
		&=
		-\lambda\cdot v|_{\mathcal{C}} - \bar{p}\int_\Gamma v
		\\
		&\leq \|\lambda\|_{\R^M} \|v\|_{0,\infty} + |\bar{p}| \|v\|_{0,1}
		\\
		&\leq C \|\phi\|_{-1,q}.
	\end{align*}
%	\marginpar{do we know in this limit that $v$ is continuous? Might expect so - density?}
	Thus we have shown that $\eta$ represents a bounded linear operator on $W^{-1,q}(\Gamma)$, thus we have shown that $-\Delta_\Gamma u -\frac{2}{R^2}u \in W^{1,q^*}(\Gamma)$.
	In particular, by Proposition \ref{app:Prop:EllReg}, it holds that $u \in W^{3,q^*}(\Gamma)$.
	% by the result of Appendix \ref{app:EllReg}.
	% elliptic regularity\marginpar{we should give a proof of this!}, which completes the result as this is true for any $p= q^*<2$.
	Since $q^*<2$ is arbitrary, the result is complete.
\end{proof}

\bibliographystyle{siam}

\begin{thebibliography}{10}

\bibitem{AdaFou03}
{\sc R.~A. Adams and J.~J. Fournier}, {\em Sobolev spaces}, Elsevier, 2003.

\bibitem{AlkDedKlo16}
{\sc M.~Alk{\"a}mper, A.~Dedner, R.~Kl{\"o}fkorn, and M.~Nolte}, {\em The
  {DUNE-ALUGrid Module.}}, Archive of Numerical Software, 4 (2016), pp.~1--28.

\bibitem{BitConFou19}
{\sc A.-F. Bitbol, D.~Constantin, and J.-B. Fournier}, {\em Membrane-mediated
  interactions}, Physics of Biological Membranes,  (2018), pp.~311--350.

\bibitem{BlaBurDed16}
{\sc M.~Blatt, A.~Burchardt, A.~Dedner, C.~Engwer, J.~Fahlke, B.~Flemisch,
  C.~Gersbacher, C.~Gr{\"a}ser, F.~Gruber, C.~Gr{\"u}ninger, et~al.}, {\em The
  distributed and unified numerics environment, version 2.4}, Archive of
  Numerical Software, 4 (2016), pp.~13--29.

\bibitem{ButNaz11}
{\sc G.~Buttazzo and S.~A. Nazarov}, {\em An optimization problem for the
  biharmonic equation with {Sobolev} conditions}, Journal of Mathematical
  Sciences, 176 (2011), p.~786.

\bibitem{Can70}
{\sc P.~B. Canham}, {\em The minimum energy of bending as a possible
  explanation of the biconcave shape of the human red blood cell}, Journal of
  {T}heoretical {B}iology, 26 (1970), pp.~61--81.

\bibitem{ChuDjuEll19}
{\sc L.~Church, A.~Djurdjevac, and C.~M. Elliott}, {\em A domain mapping
  approach for elliptic equations posed on random bulk and surface domains},
  Numerische Mathematik, https://doi.org/10.1007/s00211-020-01139-7 (2020).

\bibitem{Cia13}
{\sc P.~G. Ciarlet}, {\em Linear and Nonlinear Functional Analysis with
  Applications}, Society for Industrial and Applied Mathematics, Philadelphia,
  PA, USA, 2013.

\bibitem{DhaPer20}
{\sc S.~Dharmavaram and L.~E. Perotti}, {\em A {L}agrangian formulation for
  interacting particles on a deformable medium}, Computer Methods in Applied
  Mechanics and Engineering, 364 (2020), p.~112949.

\bibitem{DomFou02}
{\sc P.~G. Dommersnes and J.-B. Fournier}, {\em The many-body problem for
  anisotropic membrane inclusions and the self-assembly of ``saddle'' defects
  into an ``egg carton''}, Biophysical Journal, 83 (2002), pp.~2898 -- 2905.

\bibitem{DziEll13}
{\sc G.~Dziuk and C.~M. Elliott}, {\em Finite element methods for surface
  {PDE}s}, Acta Numer., 22 (2013), pp.~289--396.

\bibitem{EllFriHob17}
{\sc C.~M. Elliott, H.~Fritz, and G.~Hobbs}, {\em Small deformations of
  {H}elfrich energy minimising surfaces with applications to biomembranes},
  Math. Models Methods Appl. Sci., 27 (2017), pp.~1547--1586.

\bibitem{EllFriHob19}
{\sc C.~M. Elliott, H.~Fritz, and G.~Hobbs}, {\em Second order splitting for a
  class of fourth order equations}, Mathematics of Computation, 88 (2019),
  pp.~2605--2634.

\bibitem{EllGraHob16}
{\sc C.~M. Elliott, C.~Gr{\"a}ser, G.~Hobbs, R.~Kornhuber, and M.-W. Wolf},
  {\em A variational approach to particles in lipid membranes}, Archive for
  Rational Mechanics and Analysis, 222 (2016), pp.~1011--1075.

\bibitem{EllHat19}
{\sc C.~M. Elliott and L.~Hatcher}, {\em Domain formation via phase separation
  for spherical biomembranes with small deformations}, arXiv preprint
  arXiv:1912.10317,  (2019).

\bibitem{EllHatHer20}
{\sc C.~M. Elliott, L.~Hatcher, and P.~J. Herbert}, {\em Small deformations of
  spherical biomembranes}, in The Role of Metrics in the Theory of Partial
  Differential Equations, vol.~85 of Advanced Studies in Pure Mathematics,
  Tokyo, Japan, 2020, Mathematical Society of Japan, pp.~39--61.

\bibitem{EllHer20-A}
{\sc C.~M. Elliott and P.~J. Herbert}, {\em Second order splitting of a class
  of fourth order {PDE}s with point constraints}, Math. Comp.,  (2020).

\bibitem{ErnGue04}
{\sc A.~Ern and J.~L. Guermond}, {\em Theory and Practice of Finite Elements},
  Applied Mathematical Sciences, Springer New York, 2004.

\bibitem{FouGal15}
{\sc J.-B. Fournier and P.~Galatola}, {\em High-order power series expansion of
  the elastic interaction between conical membrane inclusions}, The European
  Physical Journal E, 38 (2015), p.~86.

\bibitem{GouBruPin93}
{\sc M.~Goulian, R.~Bruinsma, and P.~Pincus}, {\em Long-range forces in
  heterogeneous fluid membranes}, Europhysics Letters ({EPL}), 22 (1993),
  pp.~145--150.

\bibitem{Gov18}
{\sc N.~Gov}, {\em Guided by curvature: Shaping cells by coupling curved
  membrane proteins and cytoskeletal forces}, Philosophical Transactions of the
  Royal Society B: Biological Sciences, 373 (2018).

\bibitem{GraKie17}
{\sc C.~{Gr{\"a}ser} and T.~{Kies}}, {\em {On differentiability of the
  membrane-mediated mechanical interaction energy of discrete-continuum
  membrane-particle models}}, arXiv preprint arXiv:1711.1119,  (2017).

\bibitem{GraKie19}
{\sc C.~Gr\"aser and T.~Kies}, {\em Discretization error estimates for penalty
  formulations of a linearized {C}anham--{H}elfrich-type energy}, IMA Journal
  of Numerical Analysis, 39 (2019), pp.~626--649.

\bibitem{Har02}
{\sc P.~Hartman}, {\em Ordinary differential equations}, vol.~38 of Classics in
  Applied Mathematics, Society for Industrial and Applied Mathematics (SIAM),
  Philadelphia, PA, 2002.

\bibitem{Hel73}
{\sc W.~Helfrich}, {\em Elastic properties of lipid bilayers: theory and
  possible experiments}, Zeitschrift f{\"u}r Naturforschung C, 28 (1973),
  pp.~693--703.

\bibitem{HenBouMei10}
{\sc W.~M. Henne, E.~Boucrot, M.~Meinecke, E.~Evergren, Y.~Vallis, R.~Mittal,
  and H.~T. McMahon}, {\em {FCHo} proteins are nucleators of clathrin-mediated
  endocytosis}, Science, 328 (2010), pp.~1281--1284.

\bibitem{HenKenFor07}
{\sc W.~M. Henne, H.~M. Kent, M.~G. Ford, B.~G. Hegde, O.~Daumke, P.~J.~G.
  Butler, R.~Mittal, R.~Langen, P.~R. Evans, and H.~T. McMahon}, {\em Structure
  and analysis of {FCHo2 F-BAR} domain: A dimerizing and membrane recruitment
  module that effects membrane curvature}, Structure, 15 (2007), pp.~839 --
  852.

\bibitem{Kie19}
{\sc T.~Kies}, {\em Gradient methods for membrane-mediated particle
  interactions}, PhD thesis, Institut f{\"u}r Mathematik, Freie Universit{\"a}t
  Berlin, 2019.

\bibitem{KimNeuOst98}
{\sc K.~Kim, J.~Neu, and G.~Oster}, {\em Curvature-mediated interactions
  between membrane proteins}, Biophysical Journal, 75 (1998), pp.~2274 -- 2291.

\bibitem{KimNeuOst00}
{\sc K.~S. Kim, J.~Neu, and G.~Oster}, {\em Effect of protein shape on
  multibody interactions between membrane inclusions}, Phys. Rev. E, 61 (2000),
  pp.~4281--4285.

\bibitem{LarLar17}
{\sc K.~Larsson and M.~G. Larson}, {\em A continuous/discontinuous {G}alerkin
  method and a priori error estimates for the biharmonic problem on surfaces},
  Mathematics of Computation, 86 (2017), pp.~2613--2649.

\bibitem{LelRouSto10}
{\sc R.~Mathias, S.~Gabriel, and L.~Tony}, {\em Free Energy Computations: A
  Mathematical Perspective}, World Scientific, 2010.

\bibitem{Nec11}
{\sc J.~Necas}, {\em Direct methods in the theory of elliptic equations},
  Springer Science \& Business Media, 2011.

\bibitem{Reu20}
{\sc A.~Reusken}, {\em Stream function formulation of surface stokes
  equations}, IMA Journal of Numerical Analysis, 40 (2020), pp.~109--139.

\bibitem{SchKoz15}
{\sc Y.~Schweitzer and M.~M. Kozlov}, {\em Membrane-mediated interaction
  between strongly anisotropic protein scaffolds}, PLOS Computational Biology,
  11 (2015), pp.~1--17.

\bibitem{Wei18}
{\sc T.~R. Weikl}, {\em Membrane-mediated cooperativity of proteins}, Annual
  review of physical chemistry, 69 (2018), pp.~521--539.

\bibitem{WeiKozHel98}
{\sc T.~R. Weikl, M.~M. Kozlov, and W.~Helfrich}, {\em Interaction of conical
  membrane inclusions: Effect of lateral tension}, Phys. Rev. E, 57 (1998),
  pp.~6988--6995.

\bibitem{Wil65}
{\sc T.~J. Willmore}, {\em Note on embedded surfaces}, An. Sti. Univ.``Al. I.
  Cuza'' Iasi Sect. I a Mat.(NS) B, 11 (1965), pp.~493--496.

\bibitem{YolHauDes14}
{\sc C.~Yolcu, R.~C. Haussman, and M.~Deserno}, {\em The effective field theory
  approach towards membrane-mediated interactions between particles}, Advances
  in Colloid and Interface Science, 208 (2014), pp.~89 -- 109.

\end{thebibliography}

%
%\newpage
%To do:
%\begin{itemize}
%\item
%The explicit construction could be made part of a different article with Carsten and including the tube?
%\item
%Possibly include pictures for various membrane things (pictures which have been done for posters etc)
%\item
%A proof of a minimiser existing might be nice - need to check that the minimiser makes sense etc.
%\item
%Some notes on how for numerics one might use a penalty method, but the penalty method may be shown to be sufficiently close.
%\item
%Some of the sections feel like they should be subsections but they seem too disjoint
%\item
%Can we actually do this with $Z$ depending on $p$ as in \cite{Kie18}?
%\end{itemize}
\end{document}